\documentclass[12pt]{article}
\usepackage[english]{babel}
\usepackage[utf8]{inputenc}

\usepackage[nottoc]{tocbibind}
\usepackage{amsfonts, euscript}
\usepackage{changepage}
\usepackage{graphicx}
\usepackage{verbatim}
\usepackage{geometry}
\usepackage{amscd, amsmath, amssymb, amsthm}
\usepackage{appendix}
\usepackage{tikz, tikz-cd}
\usepackage{url}
\usepackage{changepage}
\usepackage{bm}
\usepackage{enumerate}
\usepackage{stmaryrd}
\usepackage{mathrsfs, calligra}
\usepackage{epigraph}
\usepackage{array}
\usepackage[english]{babel}
\usepackage{blindtext}
\usepackage{hyperref}
\usetikzlibrary{arrows}
\usepackage{tabularx}

\makeatletter
\newtheoremstyle{indented}
  {10pt}
  {10pt}
  {\addtolength{\@totalleftmargin}{3em}
   \addtolength{\linewidth}{-6em}
   \parshape 1 3em \linewidth
	 \itshape}
  {}
  {\bfseries}
  {.}
  {.5em}
  {}
\makeatother
\numberwithin{equation}{section}

\newtheorem{theorem}{Theorem}[section]
\newtheorem{corollary}[theorem]{Corollary}
\newtheorem{proposition}[theorem]{Proposition}
\newtheorem{definition}[theorem]{Definition}
\newtheorem{lemma}[theorem]{Lemma}

\newtheorem{remark}[theorem]{Remark}

\theoremstyle{definition}

\theoremstyle{remark}

\begin{document}
\medskip
\thispagestyle{empty}

\newcommand{\C}{{\mathbb C}}
\newcommand{\bS}{{\mathbb S}}
\newcommand{\Q}{{\mathbb Q}}
\newcommand{\N}{{\mathbb N}}
\newcommand{\R}{{\mathbb R}}
\newcommand{\Z}{{\mathbb Z}}
\newcommand{\A}{{\mathbb A}}
\newcommand{\F}{{\mathbb F}}
\newcommand{\f}{{\mathfrak f}}
\newcommand{\g}{{\mathfrak g}}
\newcommand{\n}{{\mathfrak n}}
\newcommand{\m}{{\mathfrak m}}
\newcommand{\X}{{\mathfrak X}}
\renewcommand{\O}{{\mathcal O}}
\newcommand{\cP}{{\mathcal P}}
\newcommand{\U}{{\mathcal U}}
\newcommand{\p}{{\mathfrak p}}
\newcommand{\T}{{\bf T}}
\newcommand{\V}{{\mathcal V}}
\newcommand{\W}{{\mathcal W}}
\renewcommand{\P}{{\mathfrak P}}
\renewcommand{\epsilon}{{\varepsilon}}
\newcommand{\xx}{^{\times}}
\newcommand{\Gal}{{\mathrm{Gal}}}
\newcommand{\GL}{{\mathrm{GL}}}
\newcommand{\Ker}{{\mathrm{Ker}}}
\newcommand{\Sym}{\mathrm{Sym}}
\newcommand{\PGL}{{\mathrm{PGL}}}
\newcommand{\PSL}{{\mathrm{PSL}}}
\newcommand{\SL}{{\mathrm{SL}}}
\newcommand{\Frob}{{\mathrm{Frob}}}
\newcommand{\im}{{\mathrm{Im}}}
\newcommand{\Stab}{{\mathrm{Stab}}}
\newcommand{\spec}{\text{Spec}}
\newcommand{\topo}{\text{top}}
\newcommand{\lie}{\text{Lie}}
\newcommand{\Id}{\text{Id}}
\newcommand{\res}{\text{Res}}
\newcommand{\Hom}{\text{Hom}}
\newcommand{\End}{\text{End}}
\newcommand{\Ind}{\text{Ind}}
\newcommand{\pr}{\text{pr}}

\title{NON-EISENSTEIN COHOMOLOGY OF LOCALLY SYMMETRIC SPACES FOR $\GL_2$ OVER A CM FIELD}

\author{SHAYAN GHOLAMI\footnote{
\includegraphics[scale=0.03]{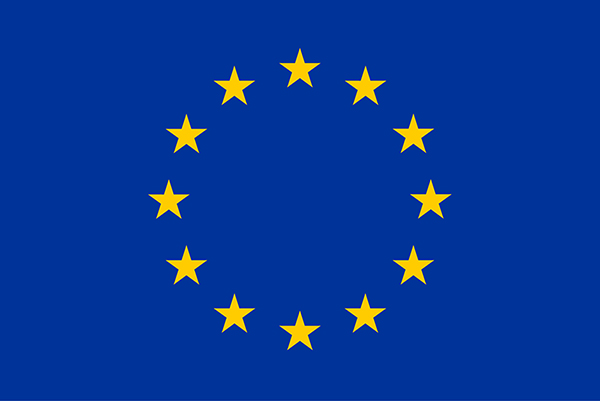} 
The author was supported by the European Union’s Horizon 2020 research and innovation program under the Marie
Sklodowska-Curie grant agreement No 754362}}
\date{}
\maketitle

\begin{abstract}
Let $F$ be a CM field, let $p$ be a prime number.
 The goal of this paper is to show, under mild conditions, that the modulo $p$ cohomology of the locally symmetric spaces $X$ for $\GL_2(F)$ with level prime to $p$
 is concentrated in degrees belonging to the Borel-Wallach range $[q_0,q_0+\ell_0]$ after localizing at a "strongly non-Eisenstein" maximal ideal of the Hecke algebra.
From this result, we deduce expected consequences on the structure of the first and last cohomology groups as modules over the Hecke algebra.
\end{abstract}
\tableofcontents
\section{Introduction}

 Let $G_{/\Q}$ be a connected reductive group and let $X=G_\infty/K_\infty$ be the symmetric space for $G_\infty=G(\R)$. 
For any neat compact open subgroup $K\subset G(\A_f)$, we can consider the locally symmetric space $X_K=G(\Q)\backslash( G(\A_f)/K\times X)$.
 
 Let us define the integers 
 $$l_0 := rk(G_\infty) - rk(K_\infty) \text{ and } q_0 := \frac{1}{2} (\dim_{\R} X  - l_0)$$
 
 A well-known conjecture (see for instance \cite[Section 9, Conjecture B (4) (a)]{CG18}) predicts that the cohomology of the locally symmetric space $X_K$ 
with any $\F_p$-local system vanishes after localizing at a non-Eisenstein maximal ideal of the Hecke algebra of $G$, 
outside of the range $[q_0,q_0+l_0]$. It was motivated by its known analogue for $\Q_p$-local systems (cf. \cite{BW} Theorem III.5.1).
 
 When $F$ is a CM field with degree $2d$ and $G = \res_{F/\Q} GL_2$, the dimension of $X$ is equal to $4d-1$, $q_0 = d$ and $l_0=2d-1$.

In this paper, we assume that the CM field $F$ is the compositum of a totally real field $F^+$ and an imaginary quadratic field $F_0$.
Let $p$ be a prime. The spherical Hecke algebra $\T=\T^S$ (defined in Section 2) is generated by $T_{v,1},T_{v,2}$ for all finite places $v\notin S$. 

Let $\O$ be the valuation ring of a $p$-adic field, let $\varpi$ be a uniformizing parameter in $\O$;
Given an algebraic representation $\V$ of $G$ over $\O$,
the Hecke algebra $\T$ acts on the Betti cohomology groups $H^i(X_K, \V)$, resp. $H^i(X_K, \V/\varpi^m \V)$.
Let $\m \subset \T$ be a maximal ideal in the support of $H^i (X_K,\V)$. 

By a theorem of Scholze \cite[Theorem I3]{Sch15}, there exists a sufficiently big finite extension $\kappa$ of $\F_p$,
and a continuous semisimple Galois representation
 
 $$ \bar{\rho}_\m : G_F \to \GL_2(\kappa)$$

which is characterized, up to conjugation, as follows:
for every $v \notin S$, $\bar{\rho}_\m$ is unramified at $v$ and the characteristic polynomial of $\bar{\rho}_\m(\Frob_v)$
is equal to $X^2 - T_{v,1}X + T_{v,2}N(v) (mod \,\m)$.

Let us say that $\overline{\rho}_\m$ has very large image if there exists a subfield $\kappa^\prime\subset \kappa$ such that
$\SL_2 (\kappa^\prime) \subset \im (\bar{\rho}_{\m}) \subset \kappa^{\times} 
\GL_2(\kappa^\prime)$.

 We write $\V=\V_\lambda$ for a representation of highest weight $\lambda=(m_\tau,n_\tau)_\tau$ for all embeddings $\tau$ of $F$.
Let $c$ be the complex conjugation. We assume that $\lambda=(m_\tau,n_\tau)_\tau$ is pure, that is:

$m_\tau+n_{\tau c}=n_\tau+m_{\tau c}$ for all $\tau$'s.

A cuspidal representation $\pi$ of $\GL_2 (\A_F)$ is called of cohomological weight $\lambda=(m_\tau,n_\tau)_\tau$ if it occurs in $H^i (X_K,\V_\lambda\otimes\C)$.
The purity weight $w=w(\lambda)$ of $\lambda$ is defined as $w=n_\tau+m_{\tau c}=n_\tau+m_{\tau c}$ for any $\tau$.

Our main result is

 \begin{theorem}
 Let $\pi$ be a cuspidal automorphic 
 representation of $\GL_2 (\A_F)$ with level $K$, cohomological weight $
 \lambda$ and purity weight $w$. Assume that $p$ splits in $F_0$, $K_p$ is hyperspecial and $\pi$ is unramified 
 outside of $S$. Let $\m \subset \T^S$ be the maximal 
 ideal of the Hecke algebra associated to $(\pi,p)$.

Assume that $w$ is even, $n_\tau \equiv n_{\tau c} \text{ (mod }2)$ for all $\tau \in \Hom(F,E)$, $p>2d^2(d+1)(-2w+4)$ and that $\overline{\rho}_\m$ has very large image.
 Then, for any $m\geq 1$, 
$$H^i (X_K, \V_\lambda/\varpi^m\V)_\m=0\quad\mbox{\rm for all}\,\,i \notin[q_0,q_0+l_0].$$
\end{theorem}  

Note that the assumption of very large image implies that $\bar{\rho}_\m$ is absolutely irreducible, in other words, that the maximal ideal $\m$ is non-Eisenstein.
To recall that the image of Galois is very large, we call our maximal ideal strongly non-Eisenstein.
Note also that we allow the weight $\lambda$ to be non-regular so that the cuspidal representation $\pi$ can be a base change from a subfield $F_1$
 of $F$ of a cuspidal representation $\pi_1$, provided the corresponding maximal ideal $\m_1$ is strongly non-Eisenstein.

\begin{remark}
By Poincar\'e duality, it suffices to prove the vanishing for $i<q_0=d$ (see Proposition \ref{poincare} in the text). 
\end{remark}

\begin{remark}
There are many vanishing results after localizing at non-Eisenstein maximal ideal for $l_0 =0$, such as: Mokrane-Tilouine \cite{MT} 
for $G=GSp_{2g}$, Dimitrov \cite{dim} $G=\res_{F/\Q} \GL_2$ where $F$ is a totally real field, Emerton-Gee \cite{EG} for $G=U(2,1)$, 
Caraiani-Scholze \cite{CS} for compact unitary Shimura varieties and $G = \res_{F/\Q}U(n,n)$ where $F$ is a totally real field.
\end{remark}

Recall that $\V_\lambda$ denotes the local system in  $\O$-modules associated to the representation $\V_\lambda$ of highest weight $\lambda$.

 \begin{corollary}
Let $\pi$ be as above, with cohomological weight $\lambda=(m_\tau,n_\tau)_\tau$ and purity weight $w$.
Let $p$ and $\m$ as above. 

Assume that $w$ is even, $n_\tau = n_{\tau c} \text{ (mod }2)$ for all $\tau \in \Hom(F,E)$, $p>2d^2(d+1)(-2w+4)$ and that $\overline{\rho}_\m$ has very large image.
 Then, for any $m\geq 1$, 
$$H^i (X_K, \V_\lambda)_\m=0\quad\mbox{\rm for all}\,\,i \notin[q_0,q_0+l_0].$$
\end{corollary}  

This corollary follows easily (see Corollary \ref{cormin}) from the main theorem. Note that $H^{q_0} (X_K, \V_\lambda)_\m$ is torsion-free, 
but it is not the case in general for $i>q_0$ (
and in particular for $i=q_0+\ell_0$). 

Finally, recall that Calegari and Geraghty introduced Conjectures A (see \cite[Section 5.3]{CG18}) and B (see \cite[Section 9.1]{CG18}) about the existence of Galois representations over localized Hecke algebras $\T^S$ and $\T^S_Q$ satisfying local global compatibilities at all primes, which we denote by (LGC). 
Furthermore they assume in Conjecture B that $H^i(X_K,\V_\lambda/\varpi)_\m=0$,
for $i\notin [q_0,q_0+\ell_0]$. This is our Main Theorem. Moreover, if $p$ splits completely in $F$, the existence of the Galois representations defined over $T^S$ and $\T^S_Q$ as desired (except (LGC)) has been established in \cite[Theorem 6.1.4]{CGH}. Note that Calegari-Geraghty established (LGC) at primes in $Q$ by \cite[Lemma 9.6]{CG18} for $n=2$ and that results towards $(LGC)$ at other primes are proven in \cite{ACC}. We deduce from \cite[Proposition 6.4]{CG18} and from our main theorem
the result:
 
\begin{corollary}
  Under the assumptions of Theorem \ref{main}, if we assume (LGC), then
  $H^{q_0+\ell_0}(X_K,\V_\lambda)_\m$ is free of rank one over $\T^S$ and 
$H^{q_0}(X_K,\V_\lambda)_\m\cong \Hom(\T^S,\O)$ as $\T^S$-module. 
 \end{corollary}

See comments before Theorem \ref{CG} for more details on the status of Conjecture B of \cite{CG18}.

Let us give a sketch of proof of the Main Theorem. First, all considerations until the end of Section 6 are valid for $G_n=\res_{F/\Q} GL_n$ for $n$ arbitrary. 
From Section 7 on, we specialize $n$ to be $2$.
We hope to generalize our result to larger values of $n$ later.

Let $X_K$ be the locally symmetric space of $G_n$ with level $K$ and $\tilde{X}_{\tilde{K}}$ be the locally symmetric space of $\widetilde{G}_n:=\res_{F/\Q}U(n,n)$  
with level $\widetilde{K}$. Note that $\tilde{X}_{\tilde{K}}$ is a PEL Shimura variety defined over $F_0$. 
We fix a non-Eisenstein maximal ideal $\m \subset \T^S$. We can define a maximal ideal of Hecke algebra of $\widetilde{G}_n$, 
$\widetilde{\m} = \bS^{-1}(\m) \subset \tilde{\T}^S$ where $\bS\colon \tilde{\T}^S\to {\T}^S$ is the Satake transform recalled in Section 2.
 
By choosing $\widetilde{K}$ and $\widetilde{\lambda}$ adapted to the choice of $K$ and $\lambda$,  we can define by Theorem 4.2.1 of \cite{ACC} 
a Hecke-equivariant embedding $H^i(X_K, \V_\lambda/\varpi)_\m \hookrightarrow H^{i}_{\partial}(\tilde{X}_{\tilde{K}},\V_{\tilde{\lambda}}/\varpi)_{\tilde{\m}}$ 
as explained in Subsection 5.2.

We shall prove for any $n$ that $H^i(X_K, \V_\lambda/\varpi)_\m = 0$ for $i <d$. 
Note that $d$ is not equal to $q_0$ for $n>2$, but happens to be equal to $q_0$ when $n=2$.
For this purpose, it is enough to prove that 
$ H^{i}_{\partial}(\tilde{X}_{\tilde{K}},\V_{\tilde{\lambda}}/\varpi)_{\tilde{\m}} = 0 $ for $i < d$.
By Nakayama's Lemma (or d\'evissage), it suffices to show $ H^{i}_{\partial}(\tilde{X}_{\tilde{K}},\V_{\tilde{\lambda}}/\varpi)[{\tilde{\m}}]= 0 $ for $i < d$. 

For this, the first step is Theorem \ref{1-dim}; it is valid for weights $\widetilde{\lambda}$ which are mildly regular 
(as defined in section 6) and sufficiently small with respect to $p$. 
This theorem states that
for $i<d$, the Galois semisimplification of the boundary cohomology
$ H^{i}_{\partial}(\tilde{X}_{\tilde{K}},\V_{\tilde{\lambda}}/\varpi)[{\tilde{\m}}]$
viewed as a $G_{F_0}$-representation where $G_{F_0}$ is the absolute Galois group of $F_0$, is a direct sum of characters. 
In our proof, we use the main result of Pink's thesis \cite{pink} and Kostant formula \cite{kos} for computing 
 the spectral sequence associated to the boundary stratification of the minimal compactification of $\tilde{X}_{\tilde{K}}$.
 Then, we use vanishing results of \cite{LS}.
Note that for $n=2$, the mild regularity is automatic.

By the generalized Eichler-Shimura relations due to Wedhorn (Theorem \ref{ES}, Section 7), the Hecke polynomial $H_{v_0}$ at an unramified place $v_0$ of 
$F_0$ annihilates $H^{\bullet}_{\partial}(\tilde{X}_{\tilde{K}},\V_{\tilde{\lambda}}/\varpi)$. Moreover, $H_{v_0}$  modulo $\widetilde{\m}$ coincides with
the characteristic polynomial 
of $T(\Frob_{v_0})$ where $T$ denotes the twisted tensor induction representation 
$$T=(\otimes \Ind_{F}^{F_0}) ((\wedge^n \bar{\rho}_{\tilde{\m}}^{\vee}) \otimes \bar{\rho}_\psi^\vee (n(n+1)/2 -2n^2)$$
where $
\bar{\rho}_{\tilde{\m}} \cong \bar{\rho}_\m \oplus \bar{\rho}_{\m}^{c,\vee} \otimes \bar{\epsilon}^{1-2n}$ 
and $\rho_\psi$ is the Galois representation associated to the central character of $\rho_{\tilde{\m}}$ ($\epsilon$ is the cyclotomic character of $F$).

Hence the characteristic polynomial of $T$ 
annihilates $H^{i}_{\partial}(\tilde{X}_{\tilde{K}},\V_{\tilde{\lambda}}/\varpi)[\tilde{\m}]$ 

From now we assume that $n=2$ and $\lambda$ satisfies the partial regularity condition (PR) (see subsection 7.2). This won't restrict our main result because
Corollary \ref{weight} 
shows that the Main Theorem for dominant weights $\lambda$ satisfying the condition (PR) implies the Main Theorem
 for any pure dominant weight $\lambda = (m_\tau ,n_\tau)_{\Hom(F,E)}$ where $n_\tau \equiv n_{\tau c} \text{ (mod } 2)$ for any $\tau \in \Hom(F,E)$. 

In subsection 7.3, we assume that $\m$ is strongly non-Eisenstein and we study the Galois representation 
$(\otimes \Ind_{F}^{F_0}) ((\wedge^2 \bar{\rho}_{\tilde{\m}}^{\vee})(-5)$ restricted to $G_{\hat{F}}$ where $\hat{F}$ 
is the normalization of $F$ in $\bar{\Q}$. We prove that if $\gcd (p-1,w+1)=1$, the restriction of 
$(\otimes \Ind_{F}^{F_0}) ((\wedge^2 \bar{\rho}_{\tilde{\m}}^{\vee})(-5)$ to $G_{\hat{F}}$ factors through:
$$H(\F_q):= \{ (M_i)_{i=1}^d \in \prod_{i=1}^d \GL_2(\F_q) \vert \exists \delta \in \det(G_F), \forall i , \det M_i = \delta \} $$  

In subsection 8.2, we observe that, by "almost incompatibility" of the Fontaine-Laffaille weights of the twisted tensor induction representation $T$'s 
and those of the boundary cohomology, there are only three characters $\chi_i$, $i=1,2,3$, of $G_{F_0}$ which can occur simultaneously 
in the representation $T$ and in the boundary cohomology groups 
$H^{i}_{\partial}(\tilde{X}_{\tilde{K}},\V_{\tilde{\lambda}}/\varpi)[\tilde{\m}]$ 
for $i<d$, viewed as $G_{F_0}$-representations. They are given explicitely as
$$\chi_1 :=(\otimes \Ind_{F}^{F_0})
((\wedge^2\bar{\rho}_{\m}^{\vee})(-5)),\quad \chi_2=(\otimes \Ind_{F}
^{F_0}) ((\wedge^2 
\bar{\rho}_{\m}^{c})(1))\, \mbox{\rm and}\,\,  \chi_3:=(\otimes \Ind_{F}^{F_0})     
(\epsilon^{\otimes -2}).$$

Actually, in subsection 3.4, we compute the Fontaine-Laffaille weights of the $G_{F_0}$-representations 
$H^{i}_{\partial}(\tilde{X}_{\tilde{K}},\V_{\tilde{\lambda}}/\varpi)$.
 By comparison of these weights to those of the characters $\chi_1 , \chi_2 , \chi_3$,  we see that the characters $\chi_2$ and $\chi_3$ 
cannot occur in  $H^{i}_{\partial}(\tilde{X}_{\tilde{K}},\V_{\tilde{\lambda}}/\varpi)[{\tilde{\m}}]$ for $i<d$. 
Therefore the only character which can occur in  $H^{i}_{\partial}(\tilde{X}_{\tilde{K}},\V_{\tilde{\lambda}}/\varpi)[{\tilde{\m}}]$ is $\chi_1$.
Its Fontaine-Laffaille weight $d \cdot w$ is the minimal Fontaine-Laffaille weight which can occur
 in $H^{i}_{\partial}(\tilde{X}_{\tilde{K}},\V_{\tilde{\lambda}}/\varpi)[{\tilde{\m}}]$.

In subsection 8.3, we compute the multiplicity of the Fontaine-Laffaille weight $d\cdot w$ inside 
$H^{i}_{\partial}(\tilde{X}_{\tilde{K}},\V_{\tilde{\lambda}}/\varpi)_{\tilde{\m}}$. We find it is equal to the
$\kappa_\p$-dimension of $H^{i+1}(\tilde{X}_{\tilde{K},\kappa_\p}^{tor},\W_{\tilde{\lambda}}^{sub})_{\tilde{\m}}$ 
($\kappa_\p = \F_p$ is the residue field of $\p\subset \O_{F_0}$).
By Serre duality,  $H^{4d-i-1} 
(\tilde{X}_{\tilde{K},\p}^{\text{tor}}, 
\W_{-2\rho_{nc}- w_{0,G} \tilde{\lambda},\p}^{\text{can}})_{\tilde{\m}^\lor}$ is dual to
$H^{i+1}(\tilde{X}_{\tilde{K},\kappa_\p}^{tor},\W_{\tilde{\lambda}}^{sub})_{\tilde{\m}}$; thus, the 
$\kappa_\p$-dimension of this coherent cohomology group is equal to the multiplicity of the Fontaine-Laffaille weight 
$d \cdot w$ in $H^{4d-i-1}(\tilde{X}_{\tilde{K}},\V_{\tilde{\mu}}/\varpi)_{\tilde{\m}^\vee}$,
where $\tilde{\mu} := (-n_\tau +2 , -m_\tau+2, m_\tau -2, n_\tau -
2)_{\bar{\tau}\in \tilde{I}}$. Now, a vanishing result of Caraiani and Scholze \cite{CS} 
shows that $H^{4d-i-1}(\tilde{X}_{\tilde{K}},\V_{\tilde{\mu}}/\varpi)_{\tilde{\m}^\vee} = 0$. Therefore, $H^{4d-i-1} 
(\tilde{X}_{\tilde{K},\p}^{\text{tor}}, 
\W_{-2\rho_{nc}- w_{0,G} \tilde{\lambda},\p}^{\text{can}})_{\tilde{\m}^\lor}$ is zero.
\\
\\
\\
\textbf{Acknowledgement.} I would like to warmly thank my advisor Jacques Tilouine for
introducing me to the subject and for many invaluable discussions.

\subsection{Notations}
In this section, we follow rather closely the notations of  \cite[Section 2]{ACC}.

Let $F=F^+F_0$ be a CM field which is the compositum of an imaginary quadratic
field $F_0$ and a totally real field $F^+$. Let $\Delta_F$ be the discriminant of $F$.
We denote by $c$ the complex conjugation of $F$. Let $d = [F^+ : \Q] $.  For a number field $K$, let $K_\A$ 
be the adele ring of $K$, let $K_f=\widehat{\Z}\otimes K$ its finite part and $K_\infty=\R\otimes K$ its infinite part so that $K_\A=K_f\times K_\infty$. 
If $H$ is an algebraic group defined over number field $K$, we put $H_\A=H(K_\A)$, $H_f=H(K_f)$ and $H_\infty=H(K_\infty)$ so that 
$H_\A=H_f\times H_\infty$. If $S$ is a set of finite places of $K$ and $U \subset H_f$, we put $U_S := \prod_{v\in S} U_v$ and $U^S := \prod_{v \notin S} U_v$.

Let $\Psi_n$ be the matrix with $1$'s on the anti-diagonal and $0$'s elsewhere and set
$$ J_n = 
\begin{pmatrix}
0 & \Psi_n \\
-\Psi_n & 0
\end{pmatrix}$$
Let $\tilde{H}$ resp. $\underline{\tilde{H}}$ the quasi-split unitary group scheme over $\O_{F^+}$ whose points on a ring $R$ are given by

$$ \tilde{H} (R) = \{ g \in GL_{2n}(R \otimes_{\O_{F^+}} \O_F ) | ^tg J_n g^c = J_n \}.$$
resp.
$$ \underline{\tilde{H}} (R) = \{ (g,c) \in GL_{2n}(R \otimes_{\O_{F^+}} \O_F ) \times R^\times | ^tg J_n g^c = c \cdot J_n \}$$

If $\bar{v}$ is a place of $F^+$ which splits in $F$, then a choice of place $v|\bar{v}$ of $F$ 
determines a canonical isomorphism $\iota_v : \tilde{H}(F^{+}_{\bar{v}}) \cong GL_{2n}(F_v)$. Denote $\tilde{G} := \res_{\O_{F^+}/\Z} \tilde{H}$ and $\underline{\tilde{G}} := \res_{\O_{F^+}/\Z} \underline{\tilde{H}}$.

We write $\widetilde{T} \subset \widetilde{B} \subset \tilde{H}$ for the subgroups consisting of, respectively, the
diagonal and the upper-triangular matrices in $\tilde{H}$. Similarly we write $H \subset P \subset \tilde{H}$ 
for the Siegel Levi and parabolic subgroups consisting of matrices which are diagonal, respectively, upper triangular by $n\times n$-blocks. 
We have $H \cong \res_{\O_F/\O_{F^+}} \GL_{n}$.  
We introduce $B=\widetilde{B}\cap H$ and we remark that $T=\widetilde{T}\cap H=\widetilde{T}$. 
Let us define $G = \res_{O_{F^+}/\Z} H$.

Let $K\subset G_f$ be a neat compact open subgroup and $X_K$ be the adelic locally symmetric space of level $K$:
$$X_K=G(\Q)\backslash G_\A/K \Q^\times_\infty K_\infty.$$

Let $\tilde{K} \subset \tilde{G}_f$ be a neat compact open subgroup of $\widetilde{G}_f$ compatible with $K$ in the sense that 
$G_f\cap \widetilde{K}=K.$
We also consider the unitary Shimura variety 
$$\widetilde{X}_{\widetilde{K}}=\widetilde{G}(\Q)\backslash\widetilde{G}_\A/\widetilde{K}\widetilde{K}_\infty.$$
and let $\underline{\tilde{K}} \subset \underline{\tilde{G}}_f$ such that $\underline{\tilde{K}} \cap \tilde{G}_f = \tilde{K}$. We also consider the unitary Shimura variety 
$$\widetilde{X}_{\underline{\widetilde{K}}}=\underline{\widetilde{G}}(\Q)\backslash \underline{\widetilde{G}}_\A/\underline{\widetilde{K}}\underline{\widetilde{K}}_\infty.$$

The real dimension of $X_K$ is equal to $d \cdot \dim GL_n(\C) - d \cdot \dim U(n) - 1 = 2 n^2 d - n^2 d -1 = n^2 d -1$ and the real dimension of $\tilde{X}_{\tilde{K}}$ is equal to $d \cdot \dim U(n,n) - d \cdot \dim (U(n) \times U(n))  =4 n^ d - 2 n^2 d = 2 n^2 d$.

We denote by $S(\widetilde{K})$ the set of places $\bar{v}$ of $F^+$ where $\widetilde{K}_{\bar{v}}$ is not hyperspecial.
Let $S^\prime$ be the set of places of $F$ above those of $S(\widetilde{K})$.

 Let $p$ be a rational prime which splits in $F_0$, say, $(p)=\p\p^c$. Let  $S_p$, resp. $\bar{S}_p$, be the set of places of $F$, resp. $F^+$ above $p$; 
We put $\bar{S}=S(\widetilde{K})\cup \bar{S}_p$ and $S=S^\prime\cup S_p$. Note that $S$ is stable by complex conjugation.
Let $E$ be a finite extension
of $\Q_p$ inside $\bar{\Q}_p$
large enough to contain the images of all the embeddings of $F$ in $\bar{\Q}_p$. 
Let $\O$ be its valuation ring, $\varpi$ a uniformizing parameter and $\kappa=\O/(\varpi)$ be its residue field. 
Note that 
$$\widetilde{G}_p=\widetilde{G}\times \Q_p=\res_{F^+/\Q}\GL_{2n}\times \Q_p.$$

Let $\widetilde{\lambda}\in  (\Z^{2n}_+)^{\Hom(F^+,E)}$ be a dominant weight for $(\widetilde{G}_p,\widetilde{B}_p,\widetilde{T}_p)$.
 Let $\lambda \in  (\Z^n_+)^{\Hom(F,E)}$ be a dominant weight for $(G,B,T)$. 

 Let $\tilde{I} \subset \Hom(F, E)$ be a subset such that
$\Hom(F, E) = \tilde{I} \amalg \tilde{I}^c$. We wright $\tilde{\tau}$ for unique element of $I$ which extends $\tau \in \Hom (F^+ , E)$ to F. 

\begin{definition} We say that the dominant weights $\widetilde{\lambda}\in  (\Z^{2n}_+)^{\Hom(F^+,E)}$ and 
$\lambda \in  (\Z^n_+)^{\Hom(F,E)}$ are corresponding if $\tilde{\lambda}_\tau = 
(-\lambda_{\tilde{\tau}c,n},-\lambda_{\tilde{\tau}c,n-1},\cdots, -\lambda_{\tilde{\tau}c,1} , \lambda_{\tilde{\tau},1}, \lambda_{\tilde{\tau},2},\cdots , \lambda_{\tilde{\tau},n})$

\end{definition}
Let $\V_\lambda$, resp. $\V_{\widetilde{\lambda}}$, be the finite free $\O$-module with algebraic action of $G_p$, resp. $\widetilde{G}_p$, of highest weight $\lambda$
resp. $\widetilde{\lambda}$. It defines a local system on $X_K$, resp. $\widetilde{X}_{\widetilde{K}}$, which we still denote $\V_\lambda$, resp. $\V_{\widetilde{\lambda}}$. 
There are {\it a priori} several choices for the modules $\V_\lambda$ and $\V_{\widetilde{\lambda}}$ 
(see corollary 1.9 of \cite{PT}). But since $p$ will be large with respect to $\lambda$, they all coincide.
\newpage

\section{Hecke algebras}
 In this section we will define Hecke algebras for $H$ and $\tilde{H}$ and important elements in these Hecke algebras which we will use continuously in this note. 
 
 \begin{definition}
  Let $G$ be a locally profinite group, and let $U \subset G$ be an open compact subgroup. We write $\mathcal{H} (G, U)$ for the
set of compactly supported, $U$-biinvariant functions $f : G \to \Z$.
\end{definition}

Let $\ell\neq p$ be a rational prime.
Let $K$ be a finite extension of $\Q_\ell$  with parameter $\varpi_K$ and residue field $\kappa$ of cardinality $q=\ell^m$. Let $\underline{G}$ be a quasi-split 
reductive group scheme over $\O_K$, 
$\underline{S}$ is a maximal $\O_K$-split torus of $\underline{G}$, $\underline{T}$ is the maximal torus which centralizes 
$\underline{S}$, and $\underline{B}$ is a Borel subgroup containing $\underline{T}$.
Let $\underline{N}$ be the unipotent radical of $\underline{B}$.  The Weyl group $W(G,T)$ acts on $\underline{T}$ by conjugation.
Let $dn$ be the Haar measure on $\underline{N}(K)$ such that $dn(\underline{N}(\O_K))=1$. It takes values in $\Z[q^{\pm 1}]$.
Let $\delta_B : \underline{T}(K) \to  q^{\Z}$ be the modulus of $B$,
given by the formula $t \mapsto |\det_K \text{Ad}_{N} (t)|_{K}$ ($|\cdot|_K$ is the absolute value of $K$ normalized by $\vert \varpi_K\vert_K=q^{-1}$)
The Satake homomorphism
 $$N\colon f \mapsto ( t \mapsto \delta_B (t)^{1/2} \int_{n \in\underline{N}(K)}  f(tn) dn)$$


gives us the isomorphism:

$$N: \mathcal{H} (\underline{G}(K), \underline{G}(\O_K)) \otimes \Z [q^{\pm 1/2}] \to  
\mathcal{H}(\underline{T}(K),\underline{T}(\O_K))^{W(G,T)}\otimes \Z[q^{\pm 1/2}]$$

If there exist a character $\chi_G \in X^\ast (G)$ such that the character $t \mapsto \delta_B (t)^{1/2} \chi_G (t)^{1/2}$ takes values in $q^{\Z}$, then we get an isomorphism 

$$N^\prime :  \mathcal{H} (\underline{G}(K), \underline{G}(\O_K)) \otimes \Z [q^{\pm 1}] \to  
\mathcal{H}(\underline{T}(K),\underline{T}(\O_K))^{W(G,T)}\otimes \Z[q^{\pm 1}]$$

given by formula $f \mapsto (t \mapsto |\chi_G(t)|^{1/2} N(f)(t)$.\\
\\
Let us to consider $G = \tilde{H} \times \O_{F^+_{\bar{v}}}$.

If $\bar{v}$ splits in $F$, then there is an canonical isomorphisms
 $\iota_v: \tilde{H}(F^+_{\bar{v}})  \cong \GL_{2n} (F_v)$ and $T(F^+_{\bar{v}}) \cong F_v^{2n}$. Therefore, The Weyl group $W(\tilde{H}_{\bar{v}},T_{\bar{v}}) = S_{2n}$ 
 and the Hecke algebra $\mathcal{H}(T(F^+_{\bar{v}}),T(\O_{F^+_{\bar{v}}}) ) \cong \Z[Y_1^{\pm1},Y_2^{\pm1},\cdots,Y_{2n}^{\pm1}]$.

If $\bar{v}$ is inert in $F$, then there is an canonical isomorphisms
 $\iota_v: \tilde{H}(F^+_{\bar{v}})  \cong U(n,n) (F_{\bar{v}})$ and $T(F^+_{\bar{v}}) \cong F_{\bar{v}}^{n}$. 
Therefore, The Weyl group $W(\tilde{H}_{\bar{v}},T_{\bar{v}}) = S_{n} \ltimes (\Z/2\Z)^n$ 
 and the Hecke algebra $\mathcal{H}(T(F^+_{\bar{v}}),T(\O_{F^+_{\bar{v}}}) ) \cong \Z[X_1^{\pm1},X_2^{\pm1},\cdots,X_{n}^{\pm1}]$.

\begin{definition}
Let $\bar{v}$ be a place of $F^+$ which is unramified in $F$.

1) Suppose that $\bar{v}$ splits in $F$. Put $\chi_{\tilde{G}} = \det^{2n-1}$, then we have 

$$ N': \mathcal{H} (\tilde{H}(F^{+}_{\bar{v}}),\tilde{H}(\O_{F^{+}_{\bar{v}}})) \otimes \Z[q_v^{\pm 1}] \cong \Z[q_v^{\pm 1}][Y_1^{\pm 1},Y_2^{\pm 1},...,Y_{2n}^{\pm 1}]^{S_{2n}} $$

For each $i = 1, . . . , 2n$, we write $\tilde{T}_{v,i}$ for the element of $\mathcal{H} (\tilde{H}(F^{+}_{\bar{v}}),\tilde{H}(\O_{F^{+}_{\bar{v}}})) \otimes \Z[q_v^{-1}]$ which corresponds
under the Satake transform to the element 
$q_v^{i(2n-i)/2} e_i(Y_1,\cdots,Y_{2n})$ ($e_i$ is $i$-th symmetric polynomial)

2) Suppose that $\bar{v}$ is inert in $F$. Since $q_{\bar{v}}=q_v^2$, the Satake isomorphism gives a canonical isomorphism

$$ N': \mathcal{H} (\tilde{H}(F^{+}_{\bar{v}}),\tilde{H}(\O_{F^{+}_{\bar{v}}})) \otimes \Z [q_{v}^{\pm 1}] \cong \Z[q_{v}^{\pm 1}][X_1^{\pm 1},X_2^{\pm 1},...,X_{n}^{\pm 1}]^{S_{n}\ltimes(\Z/2 \Z)^n}$$

 we write $\tilde{T}_{v,i}$ for the element of $\mathcal{H} (\tilde{H}(F^{+}_{\bar{v}}),\tilde{H}(\O_{F^{+}_{\bar{v}}})) \otimes \Z[q_v^{-1}]$ which corresponds
under the Satake transform to the element 
$q_v^{i(2n-i)/2} e_i(Y_1,\cdots,Y_{2n})$, where $\{Y_1,Y_2, \cdots ,\newline Y_{2n}
\} = \{ X_1^{\pm 1} , \cdots, X_n^{\pm 1} \}$.
\end{definition}

For each $\bar{v}$ of $F^+$ we have an algebra homomorphism $\mathcal{H} (\tilde{H}(F^{+}_{\bar{v}}),\tilde{H}(\O_{F^{+}_{\bar{v}}})) \to \mathcal{H} (H(F^{+}_{\bar{v}}),H(\O_{F^{+}
_{\bar{v}}}))$, we call it $\bS_{\bar{v}} = r_G \circ r_P $  unnormalized Satake transform, where algebra homomorphism $r_P: \mathcal{H} (\tilde{H}
(F^{+}_{\bar{v}}),\tilde{H}(\O_{F^{+}_{\bar{v}}})) \to \mathcal{H} (P(F^{+}
_{\bar{v}}),P(\O_{F^{+}_{\bar{v}}}))$ defines by $ f \mapsto f \vert_{P(F^+)}$ and algebra homomorphism $r_G : \mathcal{H} (P(F^{+}
_{\bar{v}}),P(\O_{F^{+}_{\bar{v}}})) \to \mathcal{H} (H(F^{+}
_{\bar{v}}), \newline H(\O_{F^{+}_{\bar{v}}}))$ defines by $ f \mapsto (x \mapsto \int_{U(F^+)} f(xu) du )$.
 
\begin{definition}
Let $\bar{v}$ be a place of $F^+$ which is unramified in $F$.

1)Suppose that $\bar{v}$ splits in $F$. The unnormalized Satake transform corresponds under the Satake
isomorphism to the map

$$\Z[q_v^{\pm 1}][Y_1^{\pm 1},Y_2^{\pm 1},...,Y_{2n}^{\pm 1}]^{S_{2n}}
\to \Z[q_v^{\pm 1}][W_1^{\pm 1},W_2^{\pm 1},...,W_{n}^{\pm 1},Z_1^{\pm 1},Z_2^{\pm 1},...,Z_{n}^{\pm 1}]^{S_{n}\times S_n}$$

where the set $\{ Y_1, \cdots, Y_{2n} \}$ is in bijection with $\{ q_v^{n/2}Z_1^{-1},\cdots , q_v^{n/2}Z_n^{-1} , q_v^{-n/2}W_1, \cdots \newline , q_v^{-n/2}W_n \}$. For all $i =1,\cdots,n$, let $T_{v,i} \in \mathcal{H} (H(F_v), H(\O_{F_v}))$ corresponds to $q_v^{i(n-i)/2} e_i(W_1,\cdots,W_{n})$.

2). Suppose instead that $\bar{v}$ is inert in $F$. Then the unnormalized Satake transform corresponds under
the Satake isomorphism to the map

$$\Z[q_v^{\pm 1}][X_1^{\pm 1},X_2^{\pm 1},...,X_{n}^{\pm 1}]^{S_{n}\ltimes(\Z/2 \Z)^n} \to \Z[q_v^{\pm 1}][W_1^{\pm 1},W_2^{\pm 1},...,W_{n}^{\pm 1}]^{S_{n}}$$

where the set $\{ X_1, \cdots, X_{n} \}$ is in bijection with $\{ q_v^{-n/2}W_1, \cdots, q_v^{-n/2}W_n \}$. For all $i =1,\cdots,n$, let $T_{v,i} \in \mathcal{H} (H(F_v), H(\O_{F_v}))$ corresponds to $q_v^{i(n-i)/2} e_i(W_1,\cdots,W_{n})$.

\end{definition}

Let $\tilde{\T}^S := \otimes_{\bar{v} \notin \bar{S}}' (\mathcal{H} (\tilde{H}(F^{+}_{\bar{v}}),\tilde{H}(\O_{F^{+}_{\bar{v}}})) \otimes \O)$, resp.$\T^S  := \otimes_{v \notin S}' (\mathcal{H} (H(F_{v}),\tilde{H}(\O_{F_{v}})) \otimes \O)$, be base change to $\O$ the abstract spherical Hecke algebra outside $S$ for $\tilde{H}$, resp. for $H$.
Let $\bS = \otimes_{v\notin S}\bS_v\colon \tilde{\T}^S\to \T^S$ be the twisted Satake homomorphism.

 Let $D(\O)$ be the derived category of complexes of $\O$-modules.
 By section 2.1.2 of \cite{ACC}, the complex $R\Gamma (\tilde{X}_{\tilde{K}}, \V_{\tilde{\lambda}})$, resp. $R\Gamma (X_K , \V_\lambda))$ has $\tilde{T}^S$-module structure, resp. $T^S$-module structure. We denote by  $\tilde{\T}^S (\tilde{K},\tilde{\lambda})$, 
resp. $\T^S (K,\lambda)$, the image of $\tilde{\T}^S$, resp.$\T^S$,
 in $
\End_{D(\O)} (R\Gamma (X_K , \V_\lambda))$, resp. in $
\End_{D(\O)} (R\Gamma (\tilde{X}_{\tilde{K}}, \V_{\tilde{\lambda}}))$.

\section{Shimura varieties}

Let  $V= \O_{F}^{2n}$ be equipped with the skew-hermitian form
$$\langle(x_1,\cdots,x_{2n}),(y_1,\cdots,y_{2n})\rangle
=\sum_{i=1}^{n} (x_i \bar{y}_{2n+1-i} -x_{2n+1-i}y_i)$$
consider the associated alternating form
$$(\cdot,\cdot) : V \times V  \to \Q : (x,y):= \text{tr}_{\O_F/\Z}\langle x,y\rangle$$

We define $\widetilde{G}$, resp. $ \underline{\widetilde{G}}$, as the group scheme representing the functor sending any ring $R$ to
$$\tilde{G}(R) = \{ g \in \GL_{O_F}(V)(R) \vert (gv,gw)=(v,w) , \forall v,w \in V \}$$
resp.
$$\underline{\tilde{G}}(R) = \{ (g,c) \in \GL_{O_F}(V)(R) \times R^\times\vert (gv,gw)=c (v,w) , \forall v,w \in V \}$$

Since all places above $p$ in $F^+$ split in $F$, there is a direct factor submodule $L\subset V \otimes \Z_p$ of rank $n$ 
such that $V \otimes \Z_p \cong L \oplus L^\lor(1)$ (this isomorphism is compatible with the alternating form).

Let us denote by $(\cdot,\cdot)_{can}:  (L \oplus L^\lor(1)) \times  (L \oplus L^\lor(1)) \to \Z_p$ 
the alternating form $((x_1,f_1),(x_2,f_2))_{can} =f_2 (x_1)-f_1(x_2)$.
\subsection{Integral model of the Shimura variety}

\begin{definition}
Let $S$ be a scheme over $\Z[\frac{1}{\Delta_F}
]$. An abelian variety with $\underline{\tilde{G}}$-structure over
$T$ is a triple $(A, \iota, \lambda)_T$ where $A$ is an abelian scheme of dimension $[F : \Q]n$
over $T$, $\iota : \O_F \to \End(A)$ is an $\O_F$-action such that $\text{Lie} A$ is free of rank $n$ over $\O_F \otimes_\Z \O_T$, and 
$\lambda : A \to A^\vee$ is a principal polarization on A whose
Rosati involution is compatible with complex conjugation on $\O_F$ via $\iota$.
\end{definition}
We fix a neat compact open subgroup $\underline{\tilde{K}}$ of $\underline{\tilde{G}_f}$.
The functor 
$$T\mapsto\{(A, \iota, \lambda)_T \mbox{\rm abelian variety with}\, \underline{\tilde{G}}-\mbox{\rm structure and level}\, \underline{\tilde{K}}\}$$
 is representable by a scheme $\mathfrak{X}_{\tilde{K}}$ which is smooth over $\Z[\frac{1}{ \underline{S}}]$ where $\underline{S}$ is 
the set of rational primes $\ell \neq p$ such that there is a place in $S$ above of $\ell$.


\begin{theorem}
There is a natural isomorphism of manifolds
$$\tilde{X}_{\underline{\tilde{K}}} \cong \mathfrak{X}_{\tilde{K}}(\C)$$
and also the map $\tilde{G} \hookrightarrow \underline{\tilde{G}}$ induces a natural map 
$\tilde{X}_{\tilde{K}} \hookrightarrow \tilde{X}_{\underline{\tilde{K}}}$, which is an open and closed immersion.
\end{theorem}
\begin{proof}
See \cite[Lemma 2.1.1 and Proposition 2.1.6]{CS}.
\end{proof}

\subsection{Automorphic vector bundle}
In the sequel, we write $\mathfrak{X}_{\tilde{K}}$ for the base change of $\mathfrak{X}_{\tilde{K}}$ to $\Z_p$.
 
Let $A \to \mathfrak{X}_{\tilde{K}}$ be the universal abelian variety for $\mathfrak{X}_{\tilde{K}}$ and let 
$\underline{H}_{dR}^1 (A/\mathfrak{X}_{\tilde{K}})$ be the relative de Rham cohomology and 
$\underline{H}^{dR}_1 (A/\mathfrak{X}_{\tilde{K}}):= \underline{\Hom}_{\O_{\mathfrak{X}_{\tilde{K}}}}(\underline{H}_{dR}^1 (A/\mathfrak{X}_{\tilde{K}}),\O_{\mathfrak{X}_{\tilde{K}}})$

We have the canonical pairing $\langle \cdot , \cdot \rangle_{\lambda} : 
\underline{H}^{dR}_{1}(A/\mathfrak{X}_{\tilde{K}})\times \underline{H}^{dR}_{1}
 (A/\mathfrak{X}_{\tilde{K}}) \rightarrow \mathscr{O}_{\mathfrak{X}_{\tilde{K}}}(1)$
defined as the composite of $(Id \times
\lambda)_{*}$ followed by the perfect pairing
$\underline{H}^{dR}_{1}(A/\X_{\tilde{K}})\times \underline{H}^{dR}_{1} (A^{\lor}/\mathfrak{X}_{\tilde{K}}) \rightarrow \mathscr{O}_{\mathfrak{X}_{\tilde{K}}}(1)$
defined by $c_1 (\mathcal{L}) \in H_{dR}^2 (A \times_{\X_{\tilde{K}}} A^\lor /\X_{\tilde{K}})$ the first Chern
class of the Poincar\'e line bundle $\mathcal{L}$ over $A\times_{\mathfrak{X}_{\tilde{K}}}A^{\lor}$.

$$  \underline{H}^{dR}_{1}(A/\mathfrak{X}_{\tilde{K}})\times \underline{H}^{dR}_{1}
 (A/\mathfrak{X}_{\tilde{K}}) \xrightarrow{(Id \times
\lambda)_{\star}} \underline{H}^{dR}_{1}(A/\X_{\tilde{K}})\times 
\underline{H}^{dR}_{1} (A^{\lor}/\mathfrak{X}_{\tilde{K}}) 
$$
$$\xrightarrow{\cup} \underline{H}^{dR}_{2}(A/\X_{\tilde{K}}) 
\xrightarrow{\cdot c_1 (\mathcal{L})} \mathscr{O}_{\mathfrak{X}_{\tilde{K}}}(1)$$

Let $\X_{\tilde{K}}^{(1)}$ be the first infinitesimal neighborhood of the diagonal image of $\X_{\tilde{K}}$
in $\X_{\tilde{K}} \times_{\Z_p} \X_{\tilde{K}}$, and let $pr_1; pr_2 : \X_{\tilde{K}}^{(1)} \to \X_{\tilde{K}}$ be the two projections. Then we
have by definition the canonical morphism $\O_{\X_{\tilde{K}}} \to \mathscr{P}^1_{\X_{\tilde{K}}/\Z_p} := pr_{1,\star} pr_2^{\star}(\O_{\X_{\tilde{K}}})$,
where $\mathscr{P}^1_{\X_{\tilde{K}}/\Z_p}$ is the sheaf of principal parts of order $1$. The isomorphism $s: \X_{\tilde{K}}^{(1)} \to \X_{\tilde{K}}^{(1)}$ over $\X_{\tilde{K}}$ swapping the two components of the fiber product then
defines an automorphism $s^\star$ of $\mathscr{P}^1_{\X_{\tilde{K}}/\Z_p}$. The kernel of the structural morphism
$str^\star :  \mathscr{P}^1_{\X_{\tilde{K}}/\Z_p} \to \O_{\X_{\tilde{K}}}$, canonically isomorphic to $\Omega^1_{\X_{\tilde{K}}/\Z_p}$ by definition, is
spanned by the image of $s^\star - Id^\star$.

An important property of the relative de Rham cohomology of any smooth morphism like $A \to \X_{\tilde{K}}$ is that, for any two smooth lifts $\tilde{A}_1 \to \X_{\tilde{K}}^{(1)}$ and $\tilde{A}_2 \to  \X_{\tilde{K}}^{(1)}$
of $A \to \X_{\tilde{K}}$, there is a canonical isomorphism $\underline{H}^1_{dR}(\tilde{A}_1/\X_{\tilde{K}}^{(1)})\cong \underline{H}^1_{dR}(\tilde{A}_2/\X_{\tilde{K}}^{(1)})$ 
lifting the identity morphism on $\underline{H}_1^{dR}(A/\X_{\tilde{K}})$.
If we consider $\tilde{A}_1 = pr_1^\star A$ and 
$\tilde{A}_2 = pr_2^\star A$, then we obtain a canonical isomorphism
$$pr_1^\star \underline{H}^1_{dR}(A/\X_{\tilde{K}}) \cong \underline{H}^1_{dR}(pr_1^\star A /\X_{\tilde{K}}^{(1)})
 \cong \underline{H}^1_{dR}(pr_2^\star A/\X_{\tilde{K}}^{(1)}) 
 \cong pr_2^\star \underline{H}^1_{dR}(A/\X_{\tilde{K}})$$
which we denote by $Id^\star$ by abuse of notation. On the other hand, the pullback by the swapping automorphism $s: \X_{\tilde{K}}^{(1)} \to \X_{\tilde{K}}^{(1)}$ defines another canonical isomorphism 

$$s^\star: pr_1^\star \underline{H}^1_{dR}(A/\X_{\tilde{K}}) \cong \underline{H}^1_{dR}(pr_1^\star A /\X_{\tilde{K}}^{(1)})
 \xrightarrow{\sim} \underline{H}^1_{dR}(pr_2^\star A/\X_{\tilde{K}}^{(1)}) 
 \cong pr_2^\star \underline{H}^1_{dR}(A/\X_{\tilde{K}})$$

The Gauss-Manin connection  $\nabla : \underline{H}^{1}_{dR}(A/\mathfrak{X}_{\tilde{K}}) 
\rightarrow \underline{H}^{1}_{dR}(A/\mathfrak{X}_{\tilde{K}}) \otimes_{\mathscr{O}_{\mathfrak{X}_{\tilde{K}}}} 
\Omega^{1}_{\mathfrak{X}_{\tilde{K}}/\Z_p}$ on $\underline{H}^{1}_{dR}(A/\mathfrak{X}_{\tilde{K}})$ is the composition

$$\underline{H}^{1}_{dR}(A/\mathfrak{X}_{\tilde{K}}) \xrightarrow{pr_1^\star} \underline{H}^{1}_{dR}(pr_1^\star A/\mathfrak{X}_{\tilde{K}}^{(1)}) \xrightarrow{s^\star - Id^\star} \underline{H}^{1}_{dR}(A/\mathfrak{X}_{\tilde{K}}) \otimes_{\X_{\tilde{K}}} \Omega^1_{\X_{\tilde{K}}/\Z_p}$$

\begin{remark}
This Gauss-Manin connection is same with the usual Gauss-Manin connection on the relative
de Rham cohomology  $\underline{H}^{1}_{dR}(A/\mathfrak{X}_{\tilde{K}})$ (cf. \cite{KO}).
\end{remark}

\begin{definition}
The principal $\tilde{G}$-bundle over $\mathfrak{X}_{\tilde{K}}$ is the $\tilde{G}$-torsor

\begin{equation}
 \mathscr{E}_{\tilde{G}} := \underline{\text{Isom}}_{\mathcal{O}_F\otimes 
\mathscr{O}_{\mathfrak{X}_{\tilde{K}}}}[(\underline{H}_{1}^{dR}(A/\mathfrak{X}_{\tilde{K}} (1)), 
\langle\cdot,\cdot\rangle_\lambda,\O_{\mathfrak{X}_{\tilde{K}}}),
((L \oplus L^{\lor} (1))\otimes_{\Z_p} \O_{\mathfrak{X}_{\tilde{K}}} , 
\langle\cdot,\cdot\rangle_\text{can}, \O_{\mathfrak{X}_{\tilde{K}}}(1))]
\end{equation}

the sheaf of isomorphisms of $\O_{\mathfrak{X}_{\tilde{K}}}$-sheaves of symplectic 
$\O_F$-modules.
\end{definition}

We can define $\mathscr{E}_{P}$ (is a $P$-torsor) which is all isomorphism such 
above by adding condition $\underline{\text{Lie}}^{\lor}_{A^{\lor}/\mathfrak{X}_{\tilde{K}}}$ maps to 
$L^{\lor}(1)\otimes_{\Z_p} \O_{\mathfrak{X}_{\tilde{K}}}$ and $\mathscr{E}_{G}$ (is a 
$G$-torsor) is all $\O_F \otimes_{\mathbb{Z}} 
\O_{\mathfrak{X}_{\tilde{K}}}$-isomorphism between 
$\underline{\text{Lie}}^{\lor}_{A^{\lor}/\mathfrak{X}_{\tilde{K}}}$ and $L^{\lor}(1)\otimes_{\Z_p} 
\O_{\mathfrak{X}_{\tilde{K}}}$.

\begin{definition}
Let $R$ be any $\Z_p$-algebra and $W\in \text{Rep}_R (\tilde{G}_p)$, we define

$$ \mathscr{E}_{\tilde{G}} (W) := (\mathscr{E}_{\tilde{G}} \otimes_{\Z_p} R) \times_{\tilde{G}_p \otimes_{\Z_p} R} W$$

Moreover, for any $W\in \text{Rep}_R (G_p)$, we define

$$ \mathscr{E}_{G} (W) := (\mathscr{E}_{G} \otimes_{\Z_p} R) \times_{G_p \otimes_{\Z_p} R} W$$

which are coherent sheaf on $\mathfrak{X}_{\tilde{K}} \otimes_{\Z_p} R$.
\end{definition}

If $W\in \text{Rep}_R (G_p)$ be free $R$-module, we say $ \mathscr{E}_{G} (W)$ is an automorphic vector bundle.

\begin{remark}
By \cite[Proposition 3.7]{LSc}, for any $\tilde{\lambda} \in (\Z_+^{2n})^\Hom(F^+,E)$, there is an idempotent element $\epsilon_{\tilde{\lambda}}$ in $\Z_p [(V \times \Z_p)^m\rtimes S_m]$ and an integer $t_{\tilde{\lambda}}$ such that

$$ \mathscr{E}_{\tilde{G}} (\V_{\tilde{\lambda}}) \cong (\epsilon_{\tilde{\lambda}})_\star \underline{H}_m^{dR} (A^m/\X_{\tilde{K}})(t_{\tilde{\lambda}})$$
where $m = |\tilde{\lambda}|$.
\end{remark}

\subsection{Toroidal compactifications of the Shimura variety}

We follow the definitions and notations of K.W. Lan's thesis \cite{lan}.

 To the Shimura variety $\mathfrak{X}_{\tilde{K}} $, one can attach a family of toroidal compactifications. Each is attached to a collection $\Sigma$ of combinatorial data adapted
to $\tilde{K}$. The corresponding
toroidal compactification is denoted $\mathfrak{X}_{\tilde{K}}^{tor}$ or $\mathfrak{X}_{\tilde{K},\Sigma}$.

Let $\cP$ be the set of maximal $\Q$-rational parabolic subgroups of $\widetilde{G}$. Let $P=MU\in \cP$.

The combinatorial data $\Sigma = \cup \Sigma_R$ adapted to the neat level $\tilde{K}$ is a collection of fans, $\Sigma_R$ , one for each rational boundary component $R$. Each
$R$ corresponds to its stabilizer $P_R$ , which i. The fan $\Sigma_R$ gives a polyhedral cone decomposition of a partial
compactification $\bar{C}_R$ of a certain cone $C_R$ inside $U_R (\R)$ that is open, convex, and self-adjoint with respect to a $\Q$-rational positive definite quadratic
form, where $U_R$ is the center of the unipotent radical of $P_R$.

 By choosing suitable $\Sigma$, these compactifications are smooth. Let $D$ be the boundary of $\mathfrak{X}_{\tilde{K}}$ in $\mathfrak{X}_{\tilde{K}}^{tor}$. It is 
a relative Cartier divisor with normal crossings over $\Z[\underline{S}^{-1}]$.

 The universal abelian scheme $A \to \X_{\tilde{K}}$ extends to a semi-abelian scheme
$A^{ext} \to \X_{\tilde{K}}^{tor}$ , the polarization $\lambda : A \to A^\lor$ extends to a prime-to-$p$
isogeny $\lambda^{ext} : A^{ext} \to (A^{ext})^\lor$
between semi-abelian schemes where $(A^{ext})^\lor$ is the dual of the semi-abelian scheme $A^{ext}$.

By proposition of 1.4.3 of \cite{lan2}, we can define a locally free $\O_{\X^{tor}_{\tilde{K}}}$-sheaf $
\underline{H}_1^{dR}(A/\X_{\tilde{K}})^{can}$ extending the $\O_{\X_{\tilde{K}}}$-sheaf $\underline{H}_1^{dR}(A/\X_{\tilde{K}})$ and
 such that it
satisfies conditions including the existence of an alternating form 
$\langle\cdot,\cdot\rangle_\lambda$ extending the Poincar\'e pairing, and endowed with a connection with logarithmic singularities extending the Gauss-Manin 
connection $\nabla$ on $\underline{H}_1^{dR}(A/\X_{\tilde{K}})$. 
Therefore we can define $\mathscr{E}_{\tilde{G}}^{can}$ and $\mathscr{E}_{G}^{can}$ as last subsection.

  For any $\Z_p$-algebra $R$ and $W \in \text{Rep}_{R}(G_p)$, there are two extension for $\mathscr{E}_G (W)$ to $\mathfrak{X}_{\tilde{K}}$ 
which is compatible with changing level and refining of by cone decomposition $\Sigma$.
The first one is the canonical extension of $\mathscr{E}_G (W)$:
$$\mathscr{E}_G (W)^{can}= (\mathscr{E}_{G}^{can} \otimes_{\Z_p} R) \times_{G_p \otimes_{\Z_p} R} W$$
and the second is the sub-canonical extension of $\mathscr{E}_G$:
$$\mathscr{E}_G 
(W)^{sub} = \mathscr{E}_G (W)^{can} \otimes_{\O_{\mathfrak{X}_{\tilde{K}}}} I_D$$
where $I_D$ is the $\O_{\X_{\tilde{K}}}$-ideal defining the relative Cartier divisor $D$.

\begin{remark}
We have an isomorphism 
$$\mathscr{E}_{\tilde{G}} (\V_{\tilde{\lambda}})^{can(sub)}
 \cong (\epsilon_{\tilde{\lambda}})_\star \underline{H}_m^{dR} (A^m/\X_{\tilde{K}})^{can(sub)}(t_{\tilde{\lambda}})$$
where $m = |\tilde{\lambda}|$. 
\end{remark}

\subsection{Local system on the integral model}
In this subsection We will define \'etale local system on $\X_{\tilde{K}} \times_{\Z[\frac{1}{\Delta_F \underline{S}}]} \bar{\Q}$ 
and we will compute Fontaine-Laffaile weight of $ H_{(c)}^i(\tilde{X}_{\tilde{K}},\V_{\tilde{\lambda}}/\varpi)$.

For any positive integer $m$, let $f_m: A_{\bar{\Q}}^m \to \X_{\tilde{K}} \times_
{\Z[\frac{1}{\underline{S}}]} \bar{\Q}$ be 
base change to $\bar{\Q}$ of the $m$-th fiber product $A^m\to \X_{\tilde{K}}$ of the universal abelian scheme.

If we fix $m= |\tilde{\lambda}|$,

 $$\V_{\tilde{\lambda},\text{\'et}} := 	(\epsilon_{\tilde{\lambda}})_\star R^m(f_m)_\star (\O)(-t_{\tilde{\lambda}})$$
 
\begin{proposition}
For any $\tilde{\lambda} \in (\Z^{2n}_+)^{\Hom(F^+,E)}$ 
there exist an \'etale local system in  $\O$-modules $
\V_{\tilde{\lambda},\text{\'et}}$ on
 $\X_{\tilde{K}} \times_
{\Z[\frac{1}{\underline{S}}]} \bar{\Q}$ such that 
it extend to $\X_{\tilde{K}} / \Z[\frac{1}{ \underline{S}}]$ and:
$$H_{\text{\'et},(c)}^i (\X_{\tilde{K}} \times_
{\Z[\frac{1}{ \underline{S}}]} \bar{\Q} , \V_{\tilde{\lambda},\text{\'et}} /\varpi^m ) \cong H_{B,(c)}^i(\tilde{X}_{\tilde{K}},\V_{\tilde{\lambda}}/\varpi^m )$$
for all positive integer $m$.
\end{proposition}

\begin{proof}
See \cite[Subsection 4.3]{LSc}.
\end{proof}

Since $\V_{\tilde{\lambda}}$ can be defined over $F_0$,  the $\O/\varpi^m$-module $H_{B,(c)}^i(\tilde{X}_{\tilde{K}},\V_{\tilde{\lambda}}/\varpi^m )$
carries a $G_{F_0}$-action.

We write $X^\star(T)^{+,P} \subset X^\star(T)^+$ for the subset of $(B \cap G)$-dominant characters. The set
$$W^P := \{ w \in W_{\tilde{G}} \vert w(X^\star (T)^+) \subset X^\star (T)^{+,P} \}$$
is the set of representatives with minimal length of the quotient $W_{\tilde{G}}/W_G$. It is called the set of Kostant representatives.

We write $\ell(w)$ for length of $w\in W^P$ and $p_{\tilde{\lambda}}(w)= [w(\tilde{\lambda}+\rho) -\rho](H)$ 
where $H = (0,0,...,0,1,1,...,1)$ (first $n$ components are $0$ and last $n$ components are $1$) and $\rho$ is half of sum of positive root of $\widetilde{G}_p$.

\begin{definition}
Let $\tilde{\lambda} \in (\Z_+^{2n})^\Hom(F^+,E)$

1) We say that $\tilde{\lambda}$ is regular if $\tilde{\lambda}_{\bar{\tau} ,i}> \tilde{\lambda}_{\bar{\tau},i+1}$ for any $\bar{\tau} \in \Hom(F^+,E)$ and $i \in [1,.2n-1]$

2)we set $|\tilde{\lambda}|_{\text{comp}} := \dim \tilde{X}_{\tilde{K}} + 1 + \sum_{\bar{\tau},i} |\tilde{\lambda}_{\bar{\tau,i}}|$. The integer
$|\tilde{\lambda}|_{\text{comp}}$ is called the comparison
size of  $\tilde{\lambda}$.
\end{definition}

Let $\text{Rep}_{\Z_p}(G_{\Q_p})$
be the category of $G_{\Q_p}$-modules of finite type over $\Z_p$ and $MF^{[0,p-2]}_{\Z_p}$ that of
finitely generated $\Z_p$-modules $M$ endowed with a filtration $(Fil^r M)_r$ such that
$Fil^r M$ is a direct factor, $Fil^0 M = M$ and $Fil^{p-1} = 0$ together with semi-linear
maps $\phi_r : Fil^r M \to M$ such that the restriction of $\phi_r$ to $Fil^{r+1} M$ is equal to
$p \phi_{r+1}$ and satisfying the strong divisibility condition : $M =\sum_{i\in \Z} \phi_r(Fil^r M)$.
Recall that by the theory of Fontaine-Laffaille \cite{FL}, we have a fully faithful functor

$$T_{cr} : MF^{[0,p-1]} \to \text{Rep}_{\Z_p} (G_{\Q_p})$$
A p-adic representation is called Fontaine-Laffaille if it is in the essential image of $T_{cr}$.

Let $\bar{X}$ be a smooth and proper scheme over $\Z_p$ of relative dimension $g$
and $D$ a relative divisor with normal crossings of $X$, we put $X = \bar{X} - D$.

\begin{theorem}[Theorem 5.3 of Faltings \cite{Fal}]\label{fal}

For any $j \in [0,p-2]$, we have following isomorphism:
$$T_{cr}(H^{j}_{log-cris,(c)}
(\bar{X}, \F_p)) \cong H^j_{\text{\'et},(c)} (X_{\bar{\Q}_p}, \F_p)$$

This isomorphism is compatible with the action of $G_{\Q_p}$. The isomorphism is functorial in
the proper smooth $\Z_p$-log scheme $\bar{X}$ and is compatible with the cup product structures
and with the formation of the Chern classes of line bundles over $\bar{X}$.
\end{theorem}
This theorem implies that for any $\tilde{\lambda} \in (\Z_+^{2n})^\Hom(F^+,E)$ such that $|\tilde{\lambda}|_{comp} < p$ we have following isomorphism

\begin{equation} \label{fal1}
T_{cr}(H^{j}_{log-cris,(c)}
(\X_{\tilde{K}}^{tor}, \V_{\tilde{\lambda}}/\varpi)) \cong H^j_{\text{\'et},(c)} (\X_{\tilde{K}} \times_{\Z_p} \bar{\Q}_p, \V_{\tilde{\lambda}}/\varpi)
\end{equation}

\begin{remark}
Since in Faltings theorem the isomorphism is compatible with the formation of chern classes, then the isomorphism \ref{fal1} is compatible with the formation of chern classes. Therefore this isomorphism is compatible with the Hecke action on the both side. 
\end{remark}
\begin{theorem}[Lan-Polo \cite{LP}]\label{LP}
Assume that $|\tilde{\lambda}|_{comp} < p$. Then the set 
of Fontaine-Laffaille weights of $H_{B,(c)}^i(\tilde{X}_{\tilde{K}},\V_{\tilde{\lambda}}/\varpi )$ 
is subset of 
the set $\{ p_{\tilde{\lambda}}(w)\vert w\in W^P , \newline
 \ell(w) \leq i \}$. 
Moreover, multiplicity of $p_{\tilde{\lambda}}(w)$ is 
equal to $\F_p$ dimensional of $H^{i-\ell (w)} 
(\X_{\tilde{K},\kappa_\p}^{tor} , \newline
 \mathscr{E}_G 
(\V_{\tilde{\lambda}_w})^{can(sub)})$.
\end{theorem}

\begin{proof}
By theorem \ref{fal} we have to consider the Hodge filteration of $H^{j}_{log-cris,(c)}
(\X_{\tilde{K}^{tor}}, \newline
 \V_{\tilde{\lambda}}/\varpi))$. Hence Theorem 5.9 of \cite{LP} implies this theorem. 
\end{proof}

\begin{corollary}\label{LP1}
Let $\tilde{\m}\subset \tilde{T}^S$ be a maximal ideal. Assume that $|\tilde{\lambda}|_{comp} < p$. Then the set 
of Fontaine-Laffaille weights of $H_{B,(c)}^i(\tilde{X}_{\tilde{K}},\V_{\tilde{\lambda}}/\varpi )_{\tilde{\m}}$ 
is subset of 
the set $\{ p_{\tilde{\lambda}}(w)\vert w\in W^P ,
 \ell(w) \leq i \}$. 
Moreover, multiplicity of $p_{\tilde{\lambda}}(w)$ is 
equal to $\F_p$ dimensional of $H^{i-\ell (w)} 
(\X_{\tilde{K},\kappa_\p}^{tor} ,
 \mathscr{E}_G 
(\V_{\tilde{\lambda}_w})^{can(sub)})_{\tilde{\m}}$.
\end{corollary}

\begin{proof}
Since the isomorphism \ref{fal1} and theorem 5.9 of \cite{LP} are compatible with the Hecke action $\tilde{T}^S$, theorem \ref{LP} implies this corollary.
\end{proof}

\subsection{Dualizing sheaf of toroidal compactification}

Generalizing previous results by Mumford, Harris and Faltings-Chai, K.-W. Lan determined the dualizing sheaf for the arithmetic toroidal compactification
of PEL Shimura varieties.  The result for $\widetilde{G}=\widetilde{G}_n$ is as follows. We fix a sufficiently fine fan so that $\X_{\tilde{K}}^{tor}$ is smooth.

\begin{theorem}
The dualizing sheaf of $\X_{\tilde{K}}^{tor}$ is isomorphic to $\mathscr{E}_G 
(\V_{-2\rho_{nc}})^{sub}$ where $\rho_{nc}$ is the half of the sum of all positive roots of $\tilde{G}$ which are not positive roots for $G$.
\end{theorem}

\begin{proof}
See Proposition 2.2.6 of \cite{har} for the rational case $\X_{\tilde{K}}^{tor} \times \Q$ and Theorem 6.4.1.1
(4) of \cite{lan} for the integral case.
\end{proof}

\begin{corollary}\label{Ser}
For any field $\kappa$ such that its characteristic does not belong to $\underline{S}$ and for any representation $W\in \text{Rep}_{\kappa}(G_p)$,
there is a Serre duality isomorphism
$$R\Hom_\kappa (R\Gamma(\X_{\tilde{K},\kappa}^{tor},\mathscr{E}_G 
(W)^{sub}),\kappa) \cong  R\Gamma(\X_{\tilde{K},\kappa}^{tor},\mathscr{E}_G 
(W^\vee \otimes \V_{-2\rho_{nc}})^{can}))$$
 Moreover this isomorphism, is compatible with the Hecke
action in the sense the action of $[\tilde{K}g\tilde{K}]$
 on the left matches the action of $[\tilde{K}g\tilde{K}]^t =
[\tilde{K} g^{-1} \tilde{K}]$ on the right.

\end{corollary}

\begin{remark}
(1) Note that for using Serre duality we need to have a Cohen-Macaulay scheme. In our case, the compactified Shimura variety $\X_{\tilde{K},\kappa}^{tor}$
 is smooth and therefore is Cohen-Macaulay. 

(2) In general case when $\tilde{K}_p$ is not hyperspecial, the compactified Shimura variety $\tilde{X}_{\tilde{K}}^{tor}$ is Cohen-Macaulay.
\end{remark}

\section{Galois representation}

Let $\pi$ be a cohomological regular cuspidal 
automorphic 
representation of $\GL_n (\A_F)$ 
with level $K$ and cohomological weight $\lambda$ 
(with pure weight $w\in\Z$). With the notations of the previous section, $\pi$ is 
unramified outside of $S^\prime$. Assume that
$K_p$ is hyperspecial so that $S^\prime\cap S_p=\emptyset$. 
Let $\m 
\subset \T^S$ be associated maximal 
ideal of Hecke algebra. By \cite{hltt} we can construct a Galois representation associated to $\m$, 
which we denote by $\bar{\rho}_\m: G_F \to \GL_n (\kappa)$, such that for any $v \notin S$, the characteristic polynomial of $\bar{\rho}_\m (\Frob_v)$ is equal to:

$$P_v (X) := X^n - T_{v,1} X^{n-1} + ... + (-1)^j q_{v}^{\frac{j(j-1)}{2}} T_{v,j} X^{n-j} + ... + (-1)^n q_{v}^{n(n-1)/2} T_{v,n} , \text{mod} \m.$$

We say that $\m$ is non-Eisenstein if $\bar{\rho}_\m$ is absolutely irreducible.

We introduce the Fontaine-Laffaille condition

$$(FL)\quad \mbox{\rm for any}\, \tau \in \Hom(F,E) \, \mbox{\rm we have}\,\,   p-2n-1 \geq \lambda_{\tau,1} + \lambda_{\tau c,1} - \lambda_{\tau,n} - \lambda_{\tau c,n}.$$

Let $\varepsilon\colon G_F\to \Z_p^\times$ be the $p$-adic cyclotomic charac and $\bar{\varepsilon}\colon G_F\to \F_p^\times$ its reduction modulo $p$.
From now on we assume $\m$ non-Eisenstein and $(FL)$.

\begin{theorem} \label{gal} Assuming the above assumptions and $K_p$ is hyperspecial, then for any $v \in S_p$, the restriction of $\bar{\rho}_\m$ to $G_{F_v}$ is Fontaine-Laffaille
 with $\tau$-Fontaine-Laffaille weight $(\lambda_{\tau ,1}+n-1, \lambda_{\tau ,2}+n-2,..., \lambda_{\tau,n})$. Also we have:

$$ \bar{\rho}_\m \vert_{I_v} \cong \begin{pmatrix}
\delta_1 & \star & \cdots & \star\\
0 & \delta_2 & \cdots & \star\\
\vdots& \vdots & \ddots & \vdots\\
0 & 0 & \cdots & \delta_n

\end{pmatrix}$$

where $\delta_1, \delta_2,..., \delta_n: I_v \rightarrow \bar{\F}_{p}^\times$
 are tame characters, whose product equals
 $\bar{\varepsilon}^{(w+{n(n-1)\over 2})}$ and whose sum has Fontaine-Laffaille
weights $(\lambda_{\tau ,1}+n-1, \lambda_{\tau ,2}+n-2,..., \lambda_{\tau,n})_{\tau \in \Hom(F_v,E)}$ .
\end{theorem}

See \cite{FL}. Note that this theorem requires the fact that for every $\tau\in \Hom(F,E)$, the Fontaine-Laffaille weights 
$\lambda_{\tau ,1}+n-1, \lambda_{\tau ,2}+n-2,..., \lambda_{\tau,n}$ are mutually distinct. 

\begin{remark} \label{llc}
If $\bar{\delta}_i: \O_{F_v} \to \bar{\F}_p^\times$ is correspondence to $\delta_i$ by local class field theory, 
then it defines by $x \mapsto \prod_{\tau \in \Hom(F_v,E)} \overline{\tau (x)}^{\lambda_\tau+n-i}$.
\end{remark}

\section{A direct summand of the boundary cohomology}

In this section, we mention three important theorems,  the first is about the occurrence of the cohomology of $X_K$ in the boundary cohomology of $\tilde{X}_{\tilde{K}}$. 
The last two ones are cohomology vanishing theorems for $\tilde{X}_{\tilde{K}}$.

From this section we denote $\tilde{\m} := \bS^{-1}(\m)$

\subsection{Embedding theorems}

In this subsection, we recall a theorem about embedding of cohomology of $X_K$ in the boundary cohomology of $\tilde{X}_{\tilde{K}}$. 
The theorem is theorem 4.2.1 of \cite{ACC}.
Let $U$ be the unipotent radical of the Siegel parabolic $P$ and $P=H\cdot U$ be a Levi decomposition of the Siegel parabolic $P$ in $\tilde{H}$.
Let us say that a compact open subgroup $\tilde{K} \subset \tilde{G}_f$ is decomposed with respect to the Siegel parabolic $P$ if
$P(F^+_f)\cap \tilde{K} =(H(F^+_f)\cap \tilde{K} )\cdot (U(F^+_f)\cap \tilde{K} )$. Let $\tilde{K}_P=P(F^+_f)\cap \tilde{K}$ and $\tilde{K}_U=U(F^+_f)\cap \tilde{K}$.

Let $\tilde{X}_{\tilde{K}}^{BS}$ be the Borel-Serre compactification of $\tilde{X}_{\tilde{K}}$ and $\tilde{X}_{\tilde{K}}^P$ be $P$-strata of the Borel-Serre compactification as subsection 2.1 \cite{ACC}. 
Let $\partial \tilde{X}_{\tilde{K}}=\tilde{X}_{\tilde{K}}^{BS}-\tilde{X}_{\tilde{K}}$ be the boundary of $\tilde{X}_{\tilde{K}}^{BS}$. 
The action of the Hecke algebra $\widetilde{T}^S$ on $R\Gamma( \tilde{X}_{\tilde{K}}, \V_{\tilde{\lambda}})$ in the derived category of $\O$-modules  $D(\O)$ 
extends canonically to $R\Gamma(\tilde{X}_{\tilde{K}}^{BS}, \V_{\tilde{\lambda}})$;
 it preserves $R\Gamma(\partial \tilde{X}_{\tilde{K}}, \V_{\tilde{\lambda}})$ and $R\Gamma(\tilde{X}^P_{\tilde{K}}, \V_{\tilde{\lambda}})$
(see \cite[Subsection 2.1]{ACC}).

By \cite[Theorem 2.4.2]{ACC}, the natural embedding $\tilde{X}_{\tilde{K}}^P \hookrightarrow \partial \tilde{X}_{\tilde{K}}$ induces an isomorphism:
$$R\Gamma(\tilde{X}_{\tilde{K}}^P ,
 \V_{\tilde{\lambda}})_{\tilde{\m}} \cong R\Gamma(\partial \tilde{X}_{\tilde{K}}, \V_{\tilde{\lambda}})_{\tilde{\m}}$$

\begin{remark}
The key ingredients for proving the above isomorphism are

1)  $
\bar{\rho}_{\tilde{\m}} \cong \bar{\rho}_\m \oplus \bar{\rho}
_{\m}^{c,\vee} \otimes \bar{\epsilon}^{1-2n}$ is the 
direct sum of two $n$-dimensional absolutely irreducible Galois representations, 

2) if $R\Gamma(\tilde{X}_{\tilde{K}}^Q ,
 \V_{\tilde{\lambda}})_{\tilde{\m}} \neq 0$ for another $\Q$-rational standard parabolic subgroup $Q$ of $\tilde{G}$, then $\bar{\rho}_{\tilde{\m}}$ 
admits another decomposition as a sum of irreducible Galois representations with respect to $Q$, which is a contradiction.
\end{remark}

Denote $\tilde{K}_P = \tilde{K} \cap p(\Q_f)$ and $\tilde{K}_U = \tilde{K} \cap U(\Q_f)$.  
Arguing in the same way as on \cite[p. 58]{NT}, we see that there is an isomorphism

\begin{equation}\label{iso}
R\Gamma(\tilde{X}_{\tilde{K}}^P ,
 \V_{\tilde{\lambda}}) \cong R\Gamma(\tilde{K}_{P}^S \times K_S, R\Gamma(\text{Inf}_{G^S \times K_S}^{P^S \times K_S} \mathbf{X}_G, R1_{\star}^{\tilde{K}_{U,S}} \V_{\tilde{\lambda}}))
\end{equation}

where $\mathbf{X}_G := G(\Q)\backslash X_G \times G(\Q_f)$ and $R1_{\star}^{\tilde{K}_{U,S}}$ is the derived functor of the functor of $\tilde{K}_{U,S}$-fixed points.
\\

The following theorem is \cite[Theorem 4.2.1]{ACC} for a special partition $(S_1,S_2)$ of the set $S_p$ of places above $p$; namely, $S_1 = S_p$ and $S_2=\emptyset$.

\begin{theorem} \label{acc1}
Let $\tilde{K} \subset \tilde{G}_f$ be a neat compact open subgroup 
which is decomposed with respect to the parabolic subgroup $P=G\cdot U$, and with the property that for each $\bar{v} \in \bar{S}_p, \tilde{K}_{U,\bar{v}} = 
U( \O_{F^{+}_{\bar{v}}} )$. Let $\tilde{\lambda} \in (\Z_{+}^{2n})^{\Hom(F^+,E)}$ and
 $ \lambda \in (\Z_{+}^n)^{\Hom(F,E)}$ be corresponding dominant weights for $\tilde{G}$ and $G$, respectively. We also assume that $p > n^2$.

Then for any $m \geq 1$, $R\Gamma (X_K, \V_{\lambda}/ 
\varpi^m)_{\m}$ is a 
$\tilde{\T}^S$-equivariant direct summand 
of $R\Gamma(\partial \tilde{X}_{\tilde{K}}, \V_{\tilde{\lambda}}/ \varpi^m)_{\tilde{\m}}$.

\end{theorem}

\begin{proof}
By using isomorphism \ref{iso}, We should show that $\V_\lambda$ is direct summand of $ R1_{\star}^{\tilde{K}_{U,S}} \V_{\tilde{\lambda}}$ as $K_S$-representation.
Note that there is a $\widetilde{K}_P$-equivariant embedding $\V_{\lambda} \to \V_{\widetilde{\lambda}}$, which splits after restriction to $K$. 

The morphism $\V_{\lambda} \to R1_\star^{\tilde{K}_{U,S}} \V_{\tilde{\lambda}}$ is the composition of the given map 
$\V_\lambda \to \V_{\tilde{\lambda}}^{\tilde{K}_{U,S}}$ with the morphism 
$\V_{\tilde{\lambda}}^{\tilde{K}_{U,S}} \to R1_\star^{\tilde{K}_{U,S}}\V_{\tilde{\lambda}}$
 whose existence is assured by the universal property of the derived functor.

The $R1_\star^{\tilde{K}_{U,S}}\V_{\tilde{\lambda}} \to \V_\lambda$ is the composition of the morphism $R1_\star^{\tilde{K}_{U,S}}\V_{\tilde{\lambda}} \to \V_\lambda$ (given by restriction to the trivial subgroup) and the $K$-equivariant splitting $\V_{\tilde{\lambda}} \to \V_{\lambda}$. This completes the proof.
\end{proof}

By the above theorem we can conclude that 

\begin{equation}\label{cohoembed}
 H^{i} (X_K, \V_\lambda / 
\varpi)_{\m} \hookrightarrow H_{\partial}^i (\tilde{X}_{\tilde{K}}, \V_{\tilde{\lambda}}/ \varpi)_{\tilde{\m}} 
\end{equation}
as $\tilde{\T}^S$-equivariant direct summand.

\subsection{Vanishings Theorems for the unitary Shimura variety}

The first theorem is the main theorem of \cite{CS} which states the vanishing of the cohomology of $\tilde{X}_{\tilde{K}}$ 
below the middle degree, after localization at certain maximal ideals $\tilde{\m}$ of the Hecke algebra.
Let $L$ be a number field with absolute Galois group $G_L$ and $\kappa$ a finite field of characteristic $p$. 
Let $\bar{r} : G_L \to \GL_n(\kappa)$ be a continuous representation.
\begin{definition}

1)  We say that a prime $\ell \neq p$ is decomposed generic for $\bar{r}$
if $\ell$ splits completely in $L$ and for all places $v \vert \ell$ of $L$, $\bar{r}\vert_{G_{L_v}}$ is unramified and the eigenvalues (with multiplicity) $\alpha_1,.., \alpha_n \in \bar{\kappa}$ of $\bar{r}(\Frob_v)$ satisfy $\alpha_i / \alpha_j \notin \{ 1, q_v \}$ for all $i \neq j$.

2) We say that $\bar{r}$ is decomposed generic if there exists a
prime $\ell \neq p$ which is decomposed generic for $\bar{r}$.
\end{definition}

\begin{theorem} \label{cs}
Assume the following conditions
\begin{itemize}
    \item[(1)] $F^+ \neq \Q$;
    \item[(2)] $\bar{\rho}_{\tilde{\m}}$ is of length at most 2;
    \item[(3)] $\bar{\rho}_{\tilde{\m}}$ is decomposed generic,
\end{itemize}

Then $H^i (\tilde{X}_{\tilde{K}}, \V_{\tilde{\lambda}}/ \varpi)_{\tilde{\m}} = 0$ for all $ i < d$ and 
$H_{c}^i (\tilde{X}_{\tilde{K}}, \V_{\tilde{\lambda}}/ \varpi)_{\tilde{\m}} = 0$ for all $i > d$.
\end{theorem}

The second theorem is the main theorem of \cite{LS}.

\begin{theorem} \label{ls}
Assume that $p > |\tilde{\lambda}|_\text{comp}$, that the weight $\tilde{\lambda}$ is regular and $\tilde{K}_p$ is hyperspecial. then
 $H^i (\tilde{X}_{\tilde{K}}, \V_{\tilde{\lambda}}/ \varpi) = 0$ for all $i < n^2 d=\dim \tilde{X}_{\tilde{K}}$.
\end{theorem}

Let us consider the long exact sequence
of the boundary:

$$ \ldots \rightarrow H_{c}^i (\tilde{X}_{\tilde{K}}, \V_{\tilde{\lambda}}/ \varpi) \rightarrow H^i (\tilde{X}_{\tilde{K}}, \V_{\tilde{\lambda}}/ \varpi) \rightarrow H_{\partial}^i (\tilde{X}_{\tilde{K}}, \V_{\tilde{\lambda}}/ \varpi) \rightarrow \ldots$$ \label{exact}

By Theorem \ref{cs}, we can conclude that 

\begin{equation} 
  H_{\partial}^i (\tilde{X}_{\tilde{K}}, \V_{\tilde{\lambda}}/ \varpi)_{\tilde{\m}}  \cong    H_{c}^{i+1} (\tilde{X}_{\tilde{K}}, \V_{\tilde{\lambda}}/ \varpi)_{\tilde{\m}} 
\end{equation}

as $\tilde{\T}^S$-module for all $i < n^2 d-1$.

Also under the assumptions of Theorem \ref{ls}, we can also conclude that 

\begin{equation} \label{ls2}
  H_{\partial}^i (\tilde{X}_{\tilde{K}}, \V_{\tilde{\lambda}}/ \varpi) \cong    H_{c}^{i+1} (\tilde{X}_{\tilde{K}}, \V_{\tilde{\lambda}}/ \varpi)
\end{equation}

as $\tilde{\T}^S$-module for all $i < n^2 d-1$ .

\section{Boundary cohomology of $\tilde{X}_{\tilde{K}}$}

Our aim in this section is to prove that for any $i<d$, the semi-simplification of
 $H_{\partial}^i (\tilde{X}_{\tilde{K}}, \V_{\tilde{\lambda}}/ \varpi )[\tilde{\m}]$ as $G_{F_0}$-module is a direct sum of characters.

Let $ j: \tilde{X}_{\tilde{K}} \hookrightarrow \tilde{X}_{\tilde{K}}^{\ast} \hookleftarrow \partial\tilde{X}_{\tilde{K}}^{\ast}:i $
where $j$, rep. $i$, is the open immersion, resp. closed. immersion of the boundary, into the minimal compactification of $\tilde{X}_{\tilde{K}}$ over $\Z[\frac{1}{\underline{S}}]$.  
It yields an isomorphism between $H_{\partial}^i (\tilde{X}_{\tilde{K}}, \V_{\tilde{\lambda}}/ \varpi)$ and 
$H^i (\partial\tilde{X}_{\tilde{K}}^{\ast}, i^{\ast} Rj_{\ast} \V_{\tilde{\lambda}}/ \varpi)$. 

We have a spectral sequence:
$$H^i (\partial\tilde{X}_{\tilde{K}}^{\ast}, i^{\ast} R^j j_{\ast} \V_{\tilde{\lambda}}/ \varpi) \Rightarrow 
H^{i+j} (\partial\tilde{X}_{\tilde{K}}^{\ast}, i^{\ast} Rj_{\ast} \V_{\tilde{\lambda}}/ \varpi)$$

 Let $P_a=M_a U_{P_a}$ be the $a$-th standard maximal parabolic subgroup of $\tilde{G}$ for any $a \in [1,n]$ where the Levi subgroup $M_a$ is isomorphic to 
$\tilde{G}_a \times \res_{\O_F/\Z} \GL_{n-a}$ and $U_{P_a}$ denotes the unipotent radical of $P_a$. Let $P_{a,h}$ be the inverse image of $\tilde{G}_a$ 
under the homomorphism $\kappa_a: P_a \to P_a/U_{P_a}=: M_a$
 Denote $M_{a,l}:= \res_{\O_F/\Z} \GL_{n-a}$ and $M_{a,h}:= \tilde{G}_a$.

 Over $\C$, we have a stratification of the boundary of the minimal compactification of the form 
$\partial\tilde{X}_{\tilde{K}}^{*} = \tilde{X}_{n-1} \coprod \tilde{X}_{n-2} \coprod ... \coprod \tilde{X}_0$ where $\tilde{X}_a$
 is a disjoint union of  $U(a,a)$-Shimura varieties $\tilde{X}_a =\coprod_{\bar{x}} \tilde{X}_{a,\bar{x}}$, 
where $x$ runs over the finite set $P(\Q)P_{a,h}(\Q_f)\backslash\tilde{G}(\Q_f)\newline/\tilde{K}$, and $\bar{x}$
 denotes an arbitrary representative of $x$ in $\tilde{G}(\Q_f)$; for later use, we may and do choose $x$ 
so that its $p$-component $x_p$ is trivial and also we have 
$\tilde{X}_{a,\bar{x}} := \tilde{G}_a(\Q)\backslash \tilde{G}_i(\Q_\A)/\tilde{K}_{a.h}^x$ where 
$\tilde{K}_a^x = x (\tilde{K}\cdot \tilde{K}_\infty) x^{-1} \cap  P_{a} (\Q_\A)$ and $\tilde{K}_{a,h}^x := \tilde{K}_a^x \cap \tilde{G}_a(\Q_\A)$.

\begin{remark}
 This stratification extends to the integral structure of Shimura varieties (cf. Theorem 7.2.4.1  of \cite{lan}).
 \end{remark}
 
This stratification on the minimal compactification gives us a stratification spectral sequence of \'etale cohomology groups (carrying Galois action of $G_{F_0}$):

\begin{equation}\label{sp1}
 E^{a-1,b}_1:= \oplus_{\bar{x}} H_{c}^{a+b-1} (\tilde{X}_{a,\bar{x}},  
Rj_{\ast} \V_{\tilde{\lambda}}/ \varpi 
\vert_{\tilde{X}_{a,\bar{x}}}) \Rightarrow H^{a+b-1} 
(\partial\tilde{X}_{\tilde{K}}^{\ast}, i^{\ast} 
Rj_{\ast} \V_{\tilde{\lambda}}/ \varpi) 
\end{equation}

This spectral sequence degenerates at $E_2$. There is also another spectral sequence of \'etale cohomology groups (also carrying Galois action of $G_{F_0}$):

\begin{equation}\label{sp2}
E_{1}^{b,c} = H_{c}^{c} (\tilde{X}_{a,\bar{x}},  
R^b j_{\ast} \V_{\tilde{\lambda}}/ \varpi 
\vert_{\tilde{X}_{a,\bar{x}}}) \Rightarrow H_{c}^{b+c} (\tilde{X}_{a,\bar{x}},  
Rj_{\ast} \V_{\tilde{\lambda}}/ \varpi \vert_{\tilde{X}_{a,\bar{x}}})
\end{equation}

This spectral sequence degenerates at $E_2$.

By the main result of \cite{pink}, the locally constant sheaf $R^b j_{\ast} \V_{\tilde{\lambda}}/ \varpi \vert_{\tilde{X}_{i,\bar{x}}}$ 
is isomorphic to the locally constant sheaf associated to the $\tilde{K}_{a,h}^x$-module

$$\bigoplus_{s=0}^bH^{b-s}(\Gamma_{M_{a,l}}(x), H^s(\Gamma_{U_{P_a}}(x),\V_{\tilde{\lambda}}/\varpi))$$

Where 
$$\Gamma_{M_{a,l}}(x) := M_{a,l}(\Q)\cap (\tilde{K}_{a,l}^x \times M_{a,l}(\Q_\infty)), \text{ for } \tilde{K}_{a,l}^x := \kappa_a(\tilde{K}^x)\cap M_{a,l}(\Q_f)$$
and
$$\Gamma_{U_{P_a}}(x):= U_{P_a}(\Q) \cap (\tilde{K}^x \cap U_{P_a}(\Q_\A))$$

Assume that $\tilde{\lambda}$ is $p$-small. Then by Theorem B of Polo and Tilouine \cite{PT}, we have for each $s\leq q$ an isomorphism of $M_a$-modules: 

$$H^s (\Gamma_{U_{P_a}}(x), \V_{\tilde{\lambda}}/\varpi) \cong \bigoplus_{w \in W^{P_a},\ell(w)= s} \V_{M_a,\tilde{\lambda}_w}/\varpi $$

where $\tilde{\lambda}_w = w(\tilde{\lambda} + \rho) - \rho$. Therefore as $\tilde{K}_{a,h}^x$-module, we have 
$$ 
H^{q-s}(\Gamma_{M_{a,l}}(x), H^s(\Gamma_{U_{P_a}}(x),\V_{\tilde{\lambda}}/\varpi)) \cong 
\bigoplus_{w \in W^{P_a},\ell(w)= s} H^{b-s}(\Gamma_{M_{a,l}}(x),\V_{M_{a,l},\tilde{\lambda}_w}/\varpi)) \otimes \V_{M_{a,h},\tilde{\lambda}_{w}} / \varpi$$

In particular, The Galois action on the cohomology of $\tilde{X}_{a,x}$ arises only from the second factors
 of each summand. Denote $V_w^{s,a} := H^s(\Gamma_{M_{a,l}}(x),\V_{M_{a,l},\tilde{\lambda}_w}/\varpi)$. 
Then we can compute the left-hand side of the spectral sequence \ref{sp2}. Namely, we have an isomorphism of $G_{F_0}$-modules:

\begin{equation} \label{dec}
H_{c}^{c} (\tilde{X}_{a,\bar{x}},  
R^b j_{\ast} \V_{\tilde{\lambda}}/ \varpi 
\vert_{\tilde{X}_{a,x}}) \cong \bigoplus_{w \in W^{P_a} , \ell(w)\leq b} H_{c}
^{c} (\tilde{X}_{a,x},   \V_{M_{a,h} ,
\tilde{\lambda}_w}/ \varpi) \otimes V_w^{b-\ell(w),a}
\end{equation}

\begin{definition}
We say $\tilde{\lambda} \in (\Z_+^{2n})^{\Hom(F^+,E)}$ is \textbf{mildly regular} if for any $\bar{\tau} \in \Hom(F^+,E)$ 
$$(\tilde{\lambda}_{\bar{\tau},2},\cdots,\tilde{\lambda}_{\bar{\tau},2n-1}) \in \Z_+^{2n-2}$$
is a regular weight.
\end{definition}

\begin{theorem} \label{1-dim}
Assume $\tilde{\lambda}$ is mildly regular and $|\tilde{\lambda}|_{comp}<p$. then any irreducible Galois sub-representation $W$ 
of $H_{\partial}^j (\tilde{X}_{\tilde{K}}, \V_{\tilde{\lambda}}/ \varpi)$ is one dimensional for all $j<d$.
\end{theorem}

\begin{proof}
We argue by induction on $n$,
 Assume that $n=1$. Since the boundary of the minimal compactification of a $U(1,1)$-Shimura variety is only 
a disjoint union of points, the boundary cohomology is the cohomology
of a zero dimensional variety and therefore is a sum of characters as Galois module.
  
The induction hypothesis being as follows:

Hypothesis: the theorem holds for all integers $k<n$.
 
 Since the spectral sequences \ref{sp1} and \ref{sp2} degenerate at $E_2$, its irreducible constituents provide the irreducible constituents of its abutment.

Let $j\in[0,d-1]$. We want to compute the boundary cohomology  $H_{\partial}^j (\tilde{X}_{\tilde{K}}, \V_{\tilde{\lambda}}/ \varpi)$.

By the spectral sequence \ref{sp1}, it is enough to check that the semisimplification of the $G_{F_0}$-modules
$H_c^j (\tilde{X}_{k,\bar{x}}, Rj_\ast\V_{\tilde{\lambda}}/ \varpi\vert_{\tilde{X}_{k,\bar{x}}})$
for the strata $\tilde{X}_{k,\bar{x}}$
for $k< n$ is a direct sum of characters. For this purpose, by the spectral sequence \ref{sp2}, it is enough to show for all $k< n$ that for every pair $(b,c)$ such that $j=b+c$, with $b+c<d$ , 
the semisimplification of $H_{c}^{c} (\tilde{X}_{k,\bar{x}},  
R^b j_{\ast} \V_{\tilde{\lambda}}/ \varpi 
\vert_{\tilde{X}_{k,\bar{x}}})$ is a direct sum of characters.

By \ref{dec}, it suffices to show that  for every $c<d$ and $k=0,\ldots,n-1$, the Galois module $H_{c}^{c} (\tilde{X}_{k,\bar{x}},   \V_{M_{k,h} ,\tilde{\lambda}_{w}}/ \varpi)$ 
for all $\bar{x}$ and $w \in W^{P_k}$,
is a direct sum of characters.
 Since the $p$-component $x_p$ of the representative $x$ is equal to $1$ by our choice, the level of the Shimura variety $\tilde{X}_{k,\bar{x}}$ at $p$ is maximal hyperspecial. 

The
local system in $\O$-modules $\V_{M_{k,h} ,\tilde{\lambda}_{w}}$ corresponds to
the representation $\V_{M_{k,h} ,\tilde{\lambda}_{w}}$ of $M_{k,h}$ which is an irreducible representation of highest weight $\tilde{\lambda}'=\tilde{\lambda}_{w,h}$.
Indeed, $\tilde{\lambda}^\prime$ is dominant since $w\in W^{P_k}$. The irreducibility follows from the condition $|\tilde{\lambda}'|_{comp} <p$.
To verify this, recall that the assumption $|\tilde{\lambda}|_{comp}<p$
guaranties that $|\tilde{\lambda}'|_{comp} <p$ because:

$$ |\tilde{\lambda}'|_{comp} = d k^2 + |\tilde{\lambda}'| \leq d k^2 + |\tilde{\lambda}_w| \leq d k^2 + |\tilde{\lambda}| + |w\rho -\rho| \leq d k^2 + |\tilde{\lambda}| $$
$$+ d(n-k)(n+k-1) 
=  |\tilde{\lambda}| + d n^2 - d (n-k) 
\leq |\tilde{\lambda}| + d n^2 = |\tilde{\lambda}|_{comp}$$ 
Note moreover that $\tilde{\lambda}^\prime$ is regular. Indeed if 
$$\tilde{\lambda}_w = w(\tilde{\lambda} + \rho) - \rho = (\tilde{\lambda}_{w,\bar{\tau},1},\cdots,\tilde{\lambda}_{w,\bar{\tau},2n})_{\bar{\tau}\in \Hom(F^+,E)}$$

then $\tilde{\lambda}^\prime$ is equal to 
$$\tilde{\lambda}^\prime = (\tilde{\lambda}_{w,\bar{\tau},k+1},\tilde{\lambda}_{w,\bar{\tau},k+2},\cdots,\tilde{\lambda}_{w,\bar{\tau},n-k})_{\bar{\tau}\in \Hom(F^+,E)}$$
 since $\tilde{\lambda}$ is mildly regular, $\tilde{\lambda}^\prime$ is regular.

Therefore, Theorem \ref{ls} implies that for any $r<d$, we have
$H_{c}^{r} (\tilde{X}_{k,\bar{x}},   \V_{M_{k,h} ,\tilde{\lambda}_{w}}/ \varpi) \cong H_{\partial}^{r-1} (\tilde{X}_{k,\bar{x}},   \V_{M_{k,h} ,\tilde{\lambda}_{w}}/ \varpi)$. 
By the induction assumption, since $k<n$, the semi-simplification $H_{\partial}^{r-1} (\tilde{X}_{k,\bar{x}},   \V_{M_{k,h} ,\tilde{\lambda}_{w}}/ \varpi)$ is a direct sum of some characters.
\end{proof}

\begin{remark}
The semisimplification is needed in this theorem because of the spectral sequences \ref{sp1} and \ref{sp2} despite the fact 
that there is no need of semisimplification in formula \ref{dec}. 
\end{remark}

\section{Eichler-Shimura relations}

We first recall a result by Wedhorn \cite{wed} establishing Eichler-Shimura relations for the group $\widetilde{G}$ relevant to our purpose 
(He proves the Eichlerâ€“Shimura relation 
in the more general case of PEL
Shimura varieties at places of good reduction at which the group is split).

\subsection{Tensor induction}
In this subsection we recall the notion of tensor induction as defined in Section $1$ of Yoshida \cite{yos}.

Let $\rho_0: G_F \to \GL(V_0)$ be a representation.	We fix a decomposition $G_{F_0} = \coprod_{i=1}^d g_i G_F$ where $g_1,\cdots,g_d \in G_{F_0}$ are all equivalence classes of $G_{F_0}/G_F$. Denote $V := \otimes_{i=1}^d V_i$ where all $V_i$ is isomorphic to $V_0$. With this choice of representatives $(g_i)$'s, we define a 
map $\rho: G_{F_0} \to \GL(V)$ by:

$$\rho(h)(\bigotimes_{i=1}^{d} v_i) = \bigotimes_{i=1}^d \rho_0 (g_i^{-1}h g_{h,i})(v_{h,i})$$

where $g_{h,i} = g_k $ for some $k = 0, \cdots ,d$ such that $h^{-1}g_i \in g_{h,i} G_F$ and $v_{h,i}=v_k$. The map $\rho$ is a group homomorphism.
 If we choose another set of representatives $(g_i^\prime)$ the resulting homomorphism is equivalent to $\rho$. We denote by $\otimes \Ind_{F}^{F_0} \rho_0$
 any representation $\rho$ as above. We call it a tensor induction of $\rho_0$ from $G_F$ to $G_{F_0}$.

\subsection{A theorem of Wedhorn}

Let $v_0$ be an arbitrary finite place of $F_0$, denote $H_{v_0}(x)$ for the characteristic polynomial 
of $[(\otimes \Ind_{F}^{F_0}) ((\wedge^n \bar{\rho}_{\tilde{\m}}^{\vee}) \otimes \rho_\psi^\vee (n(n+1)/2 -2n^2)](\Frob_{v_0})$ 
where $\psi$ is central character of $\rho_{\tilde{\m}}$ (this is the Hecke polynomial for $U(n,n)/F^+$, see \cite[A5.7]{nek}).

\begin{theorem}[Wedhorn] \label{ES}
For any finite place $v_0$ of $F_0$, the endomorphism 
 $H_{v_0}(\Frob_{v_0})$ annihilates 
$ H_{(c)}^i (\tilde{X}_{\tilde{K}}, \V_{\tilde{\lambda}}/ \varpi)[\tilde{\m}]$.
\end{theorem}

\textbf{Comment}:  the proof requires that the finite place $v_0$ of $F_0$ is above a rational prime $\ell$ which is totally split in $F$. However, 
Chebotarev density theorem implies the result for any finite place $v_0$ of $F_0$.
\\

Therefore, the long exact sequence \ref{exact} implies that  $H_{v_0}(\Frob_{v_0})$ annihilates 
$\bar{\rho}_{\partial}:= H_{\partial}^i (\tilde{X}_{\tilde{K}}, \V_{\tilde{\lambda}}/ \varpi)[\tilde{\m}]$.

From now, we fix $n=2$. We can therefore write $\lambda = (m_\tau , n_\tau)_{\tau \in \Hom(F,E)}$. We make the following assumption (PR)
 of purity and partial regularity:

\begin{itemize}
\item[(PR)] For all 
$\tau \in \Hom(F,E)$,
we have 
$m_\tau = m_{\tau^c}$, $n_\tau = n_{\tau^c}$ and $0> m_\tau$.
\end{itemize} 
The assumption (PR) implies that the central character $\psi$ of $\rho_{\tilde{\m}}$ is trivial.

For the remaining of this section, we will study $ \bar{\theta}:=(\otimes\Ind_{F}^{F_0}) (\wedge^2 \bar{\rho}_{\tilde{\m}}^{\vee} (-5))]$.
 By the definition of $\tilde{\m}$, we have an isomorphism between Galois representations 
$\bar{\rho}_{\tilde{\m}} \cong \bar{\rho}_{\m} \oplus \bar{\rho}_{\m}^{c,\vee}(-3)$.
Hence, $\wedge^2 (\bar{\rho}_{\tilde{\m}}^{\vee}) 
$ is isomorphic to $(\wedge^2\bar{\rho}_{\m}
^{\vee}) \oplus (\wedge^2 \bar{\rho}_{\m}^{c})(6) 
\oplus (\bar{\rho}_{\m}^{\vee} \otimes \bar{\rho}
_{\m}^{c})(3) $. Since $\bar{\rho}_{\m}^{c} \cong 
\bar{\rho}_{\m}^{\vee} \otimes \det \bar{\rho}
_{\m}$ therefore, $\wedge^2 (\bar{\rho}
_{\tilde{\m}}^{\vee}) $ is isomorphic to:
 
 \begin{equation} \label{theta}  
(\wedge^2\bar{\rho}_{\m}^{\vee}) \oplus (\wedge^2 
\bar{\rho}_{\m}^{c})(6) \oplus     \epsilon^3 \oplus (\Sym^2 
\bar{\rho}_{\m}^{\vee}) \otimes (\det \bar{\rho}_{\m})(3)
\end{equation}

\subsection{Large image conditions}

For the rest of this paper this paper we assume that:

\begin{itemize}
\item[(LI $\bar{\rho}_{\m}$)] there exists a power $q$ of $p$ such that $
\SL_2 (\F_q) \subset \im (\bar{\rho}_{\m}) \subset \kappa^{\times} 
\GL_2(\F_q)$.

\end{itemize}

\begin{remark}
The condition (LI $\bar{\rho}_{\m}$) implies that $\bar{\rho}_{\m}$ is decomposed generic. (for example see \cite[Lemma 2.3]{AN})
\end{remark}

Let $\hat{F}$ be normal closure of $F$ in $\bar{\Q}$. Denote $pr: \PGL_2(\bar{\F}_p) \to \GL_2(\bar{\F_p})$ for the projection map. 
Let $\mathcal{D} = \det \bar{\rho}_\m (G_{\hat{F}})$. 

\begin{lemma} \label{image}
Assume (LI $\bar{\rho}_\m$) and $p > w$. then either

$$\bar{\rho}_\m(G_{\hat{F}}) = \GL_2(\F_q)^\mathcal{D} := \{ A \in \GL_2(\F_q) \vert \det (A) \in \mathcal{D} \}, \text{ or}$$

$$  \bar{\rho}_\m(G_{\hat{F}}) = (\F_{q^2}^\times \GL_2(\F_q))^\mathcal{D} = \{ A \in \F_{q^2}^\times \GL_2(\F_q) \vert \det (A) \in \mathcal{D} \}$$
\end{lemma}

\begin{proof}

By (LI $\bar{\rho}_\m$) the group $pr (\bar{\rho}_\m(G_{F}))$ 
is isomorphic to $\PGL_2(\F_q)$ or $\PSL_2(\F_q)$. The group 
$pr( \bar{\rho}_\m(G_{\hat{F}}))$ is normal subgroup of $pr( 
\bar{\rho}_\m(G_{F}))$. Since $p>w$ and $ pr( \bar{\rho}_
\m(I_v)) \subset pr( \bar{\rho}_\m(G_{\hat{F}}))$ then theorem \ref{gal} implies that $ 
pr(\bar{\rho}_\m(G_{\hat{F}}))$ is non-trivial. 

Since $\PSL_2(\F_q)$ is a simple group of index $2$ in the group $\PGL_2(\F_q)$, we have $\PSL_2(\F_q) \subset pr( \bar{\rho}_\m(G_{\hat{F}}))$. 
In the next lemma, we show that this implies $\SL_2(\F_q) \subset \bar{\rho}_\m(G_{\hat{F}})$. So we are done by the following chain of inclusions:

$$(\kappa^\times \GL_2(\F_q))^{\mathcal{D}} \subset \bar{\rho}_\m(G_{\hat{F}}) \subset  (\GL_2(\F_q))^{\mathcal{D}} $$

\end{proof}

\begin{lemma}\label{group}
Let $H$ be a group of center $Z$ and let $pr : H \to H/Z$ the canonical projection. Let $P$ and $Q$ be two subgroups of $H$ 
such that $pr(Q) \subset pr(P)$. Assume moreover
that Q has no non-trivial abelian quotients. Then $Q \subset P$.
\end{lemma}

\begin{proof}
Since $Q$ has no non-trivial abelian quotients, $Q^{der}$ is equal to $\{ 1\}$ or $Q$. If $Q^{der} = \{ 1 \}$, then $Q$ is an abelian group
, contradicting the assumption. So $Q^{der}=Q$ and we have following a chain of reverse inclusions:

$$ P \supset P^{der} = (PZ)^{der} = (QZ)^{der} = Q^{der} = Q $$ 
\end{proof}

Let $\gamma \in \F_{q^2} - \F_q$ be such that $\gamma^2 \in \F_q$.
 Then:
  $$(\GL_2(\F_{q^2}))^{\mathcal{D}} =
  (\GL_2(\F_{q}))^{\mathcal{D}} \coprod (\gamma
  \GL_2(\F_{q}))^{\mathcal{D}}$$

Since $\pr(\Ind_{F}^{F_0} \bar{\rho}_\m(G_{\hat{F}})) \subset 
\prod_{\tau \in \Hom(F,E)}\PGL_2(\F_q)$ by using 
Goursat's lemma as in Lemma 5.2.2 of \cite{ribet} and 
the facts $\PSL_2(\F_q)$ is a simple group and $\text{Aut}(\PSL_2(\F_q)) = \PSL_2(\F_q) \rtimes \Gal(\F_q/\F_p)$, we see that there exists a partition  
$\Hom(F,E) = 
\coprod_{i\in I} J^{i}_F$
and for all $\tau \in 
J^{i}_F$ there exist an element
$\sigma_{i,\tau} \in \Gal(\F_q/\F_p)$ such that

$$\pr(\Phi(\SL_2(\F_q)^I)) \subset \pr(\Ind_{F}^{F_0} \bar{\rho}_\m(G_{\hat{F}}))\subset \pr(\Phi(\GL_2(\F_q)^I)),$$\label{serre}

Where $\Phi = (\Phi_i)_{i\in I} : \GL_2 (\F_q)^I \rightarrow 
\GL_2 (\F_q)^{\Hom(F,E)}$ is given by $(M_i)_{i\in 
 I} \mapsto (M^{\sigma_{i,\tau}}_i )_{i\in I,
 \tau \in J^{i}_F}$ ($\pr : \GL_2 \rightarrow \PGL_2$ is 
 projection map). 
 

Denote $J_{F,v} := \Hom(F_v,E)$. Let 

$$H(\F_q):= \{ (M_i)_{i \in I} \in \prod_{i\in I} \GL_2(\F_q) \vert \exists \delta \in \mathcal{D}, \forall i , \det M_i = \delta \} $$.
\begin{lemma} \label{inertia}
Assuming (LI $\bar{\rho}_\m$) holds and $p > 2w+2$, then:
$$
\Ind_{F}^{F_0} \rho_\m (I_\p) \subset \Phi (H(\F_q))$$
\end{lemma}

\begin{proof}
Let $h= |J_{F,v}|$ or $2 |J_{F,v}|$ be the common tame level of the characters $\delta_1, \delta_2$ (which are defined in Theorem \ref{gal}) 
and $x_h$ be a generator of $\F_{p^h}^\times$. Let $\tilde{\delta}_1 , \tilde{\delta}_2 : \F_{p^h}^\times \to \bar{\F_p}^\times$ 
be the characters corresponding to $\delta_1 , \delta_2 : I_v \to \bar{\F_p}^\times$  by local class field theory. By Lemma \ref{image}, 
the all of array of the elements of $\bar{\rho}_\m(I_v)$ are in $\F_q \amalg \gamma \F_q$. Then $\tilde{\delta}_1(x_h)^2, \tilde{\delta}_2 (x_h)^2 \in \F_q$. Since 
$\tilde{\delta}_i (x_h) =  \prod_{\tau \in J_{F,v}} \overline{\tau} (x_h)^{m_\tau +1} = x_h^{\sum (m_i+1)p^i}$ or 
$\prod_{\tau \in J_{F,v}} \overline{\tau} (x_h)^{n_\tau} =  x_h^{\sum n_i p^i}$ where $n_i$ resp. $m_i$ is equal to $n_\tau$,
 resp. $m_\tau$ for $\bar{\tau}(x_h) = x_h^{p^i}$. If $q=p^l$ then $ \F_q \cap \F_{p^h} = \F_{p^{\gcd(h,l)}}$, 
we put $l' := \gcd(l,h)$ and $h = l' \cdot k$. 

Since $ \tilde{\delta}_1(x_h)^2, \tilde{\delta}_2 (x_h)^2 \in \F_q$, we have $ \tilde{\delta}_1(x_h)^2, \tilde{\delta}_2 (x_h)^2 \in \F_{p^{l'}}$,
 hence $p^h - 1 \vert 2 (\sum n_i p^i) (p^{l'}\newline-1)$ and $p^h - 1 \vert 2 (\sum (m_i +1) p^i) (p^{l'}-1)$, 
Hence $\sum_{i=0}^{k-1} p^{i\cdot l'} \vert   \sum_{i=0}^{h-1}  2 n_i p^i$ and
 $\sum_{i=0}^{k-1} p^{i\cdot l'} \vert   \sum_{i=0}^{h-1} \newline 2 (m_i +1) p^i$. Since $p > 2w +2$,
 the quotient of these numbers is equal to $ \sum_{i=0}^{l'-1} 2 n_i p^i$ resp.  $ \sum_{i=0}^{l'-1} 2 (m_i+1) p^i$ 
which are even numbers and implies that $p^h - 1 \vert (\sum n_i p^i) (p^{l'}-1)$ and $p^h - 1 \vert (\sum (m_i+1) p^i) (p^{l'}-1)$,
 then $\tilde{\delta}_1(x_h), \tilde{\delta}_2 (x_h) \in \F_q$. It implies that $\bar{\rho}_\m (I_v)$ is a subgroup of $\GL_2(\F_q)$.


Since $p$ is unramified in $F$ then for any $y \in I_{\p}$ we know $y$ acts on $\Ind_{F}^{F_0} \bar{\rho}_\m$ by $(\bar{\rho}_\m (g_i^{-1} y g_i))_{i}$. 
So for prove of lemma, it is enough to check that the determinant condition and it is true because $\bar{\rho}_\m (I_{\p})$ is subgroup of $\GL_2(\F_q)$ 
and theorem \ref{gal}.
\end{proof}

\begin{lemma}
Assuming that (LI $\bar{\rho}_\m$) holds and $p > 2w+2$, then
	
	$$H(\F_q) \subset \Ind_{F}^{F_0}\bar{\rho}_\m (G_{\hat{F}})$$
\end{lemma}
\begin{proof}
By \ref{serre} we have:
 $$\pr(\Phi(\SL_2(\F_q)^I)) \subset \pr(\Ind_{F}^{F_0} \bar{\rho}_\m(G_{\hat{F}}))$$
 Therefore lemma \ref{group} implies that $(\Phi(\SL_2(\F_q))^I) \subset (\Ind_{F}^{F_0} \bar{\rho}_\m(G_{\hat{F}}))$.
 
Then the lemma deduces by the fact $\Phi(H(\F_q)) = \Phi ((\SL_2(\F_q))^I) \cdot \Ind_{F}^{F_0}(\bar{\rho}_\m)(I_v)$.
\end{proof} 

 Denote the fixed field of $(\Ind_{F}^{F_0} \bar{\rho}_\m)^{-1} (\Phi(H(\F_q)))$ by $F'$.
 
 \begin{lemma} \label{fac}
 Assume that (LI $\bar{\rho}_\m$) holds and $\text{gcd}(p-1,w+1)=1$, then the restrictions  to $G_{F'}$ of the 
$G_{F_0}$-representations $\bar{\theta}$ and 
$\bar{\rho}_\partial^{ss}$ factor through $H(\F_q)$.
 \end{lemma}

 \begin{proof}
 Note that the condition $\text{gcd}(p-1,w+1)=1$ implies that the mod $p$ cyclotomic character factors by $\im(\bar{\rho}_\m)$.  
Then, the following diagram implies that the restriction of $\bar{\theta}$ to $F'$ factors by $H(\F_q)$
 
\[ \begin{tikzcd}
G_{\hat{F}}  \arrow{r}{\Ind_{F}^{F_0} \bar{\rho}_{\tilde{\m}}}\arrow[swap]{d} & \GL_{6}(\kappa)^{\Hom(F,E)} \arrow{d}{\otimes} \\%
G_{F_0} \arrow{r}{\otimes-\Ind_{F}^{F_0} \bar{\rho}_{\tilde{\m}}}& \GL_{6^d}(\kappa)
\end{tikzcd}
\].

Also, by theorem \ref{ES},  $\text{char}_{\bar{\theta}(g)}(\bar{\rho}_\partial (g)) = 0$ for any $g \in G_{F'}$. 
By theorem \ref{1-dim}, the semi-simplification of $\bar{\rho}_\partial$ is a direct sum of characters, so $\bar{\theta} (g) = 1$ implies $\bar{\rho}_\partial^{ss} (g) = 1$.
Hence, $ \Ker\, \bar{\theta}\subset\Ker\, \bar{\rho}_\partial^{ss}$;
therefore, $\bar{\rho}_\partial^{ss}$ factors by $H(\F_q)$.
 \end{proof}
 
 \section{Proof of the Main Theorem}
The aim of this section is to prove our main results on the cohomology of locally symmetric space of $\GL_2(F)$
 after localisation at a strongly non-Eisenstein maximal ideal.\\

\subsection{Reduction from an arbitrary weight to the (PR) case}
Let $\lambda = (m_\tau , n_\tau)_{\tau \in \Hom(F,E)}\in (\Z^2_+)^{\Hom(F,E)}$ be an arbitrary dominant weight (so $m_\tau\geq n_\tau$ for  all $\tau$'s).

If $\psi : G_F \to \O^\times$ is a continuous character,
we define an automorphism $f_\psi$ of the Hecke algebra  $\T^S$ (defined in section 2) by the formula
$f_\psi([K^S g K^S]) = \psi(\text{Art}_F (\det(g)))^{-1}[K^SgK^S]$ for all $g\in G_f$ such that $g_S=1$.  If $\m \subset \T^S$
is a maximal ideal, then we define $\m(\psi) = f_\psi(\m)$.

\begin{proposition} \label{twist}
Let $\psi : G_F \to \O^\times$ be a continuous character satisfying the following conditions:

($1$) For each $v \in S_p$, $\psi$ is unramified at $v$.

($2$) There is $m^\prime= (m^\prime_\tau )_{\tau} \in \Z^{\Hom(F,E)}$
such that for each place $v\in S_p$, and
for each $k \in \O_{F_v}^\times$, we have

$$\psi(\text{Art}_{F_v}(k)) = \prod_{\tau \in \Hom(F_v,E)} \tau (k)^{-m^\prime_\tau}$$

Let $\mu \in (\Z^2_+)^{\Hom(F,E)}$
be the dominant weight defined by the formula $\mu_\tau =
(m^\prime_\tau, m^\prime_\tau)$ for each $\tau \in \Hom(F, E)$. Then for any $\lambda \in (\Z^2_+)^{\Hom(F,E)}$
there is an isomorphism

$$R\Gamma(X_K,\V_\lambda)_\m \cong R\Gamma(X_K,\V_{\lambda+\mu})_{\m(\psi)}$$

in $D(\O)$ which is equivariant for the action of $\T^S$ when $\T^S$ acts in the usual
way on the left-hand side and acts by $f_\psi$ on the right-hand side.
\end{proposition}

\begin{proof}
See Proposition 2.2.14 and Corollary of 2.2.15 of \cite{ACC}. \end{proof}

There is an involution $\iota : \T^S \to \T^S$, resp. $\tilde{\iota} : \tilde{\T}^S \to \tilde{\T}^S$, which sends a double coset $[K^S g K^S] \mapsto [K^S g^{-1}K^S]$, resp.  $[\tilde{K}^{\bar{S}} g \tilde{K}^{\bar{S}}] \mapsto [\tilde{K}^{\bar{S}} g^{-1} \tilde{K}^{\bar{S}}]$ . If $\m \subset \T^S$, resp. $\tilde{\m} \subset \tilde{\T}^S$,
is a maximal ideal with residue field a finite
extension of $\kappa$, then we define $\m^\vee = \iota(\m)$, resp. $\tilde{\m}^\vee = \iota(\tilde{\m})$.

\begin{proposition}\label{poincare}
Let $\lambda=(m_\tau,n_\tau)_{\tau} \in (\Z^2_+)^{\Hom(F,E)}$, resp $\tilde{\lambda}=(\tilde{\lambda}_{\bar{\tau},1},\cdots,\tilde{\lambda}_{\bar{\tau},4})_{\tau} \in (\Z^4_+)^{\Hom(F^+,E)}$. Denote $\lambda^\lor:=(-n_\tau,-m_\tau)_\tau$, resp. $\tilde{\lambda}^\lor:=(-\tilde{\lambda}_{\bar{\tau},4},\cdots,-\tilde{\lambda}_{\bar{\tau},1})_{\bar{\tau}}$. Then there is an isomorphism

$$R\Hom_\O (R\Gamma(X_K,\V_\lambda)_\m,\O) \cong R\Gamma_c(X_K,\V_{\lambda^\lor})_{\m^\lor}[4d-1]$$
rsep.
$$R\Hom_\O (R\Gamma(\tilde{X}_{\tilde{K}},\V_{\tilde{\lambda}})_{\tilde{\m}},\O) \cong R\Gamma_c(\tilde{X}_{\tilde{K}},\V_{\tilde{\lambda}^\lor})_{\tilde{\m}^\lor}[8d]$$
in $D(\O)$ which is equivariant for the action of $\T^S$, resp. $\tilde{\T}^S$, when $\T^S$, resp. $\tilde{\T}^S$, acts by $\iota$, resp, $\tilde{\iota}$, on the
left-hand side and in its usual way on the right-hand side.

\end{proposition}

\begin{proof}
See Proposition 3.7 of \cite{NT}.
\end{proof}

\begin{corollary}\label{weight}
Let $\lambda=(m_\tau,n_\tau)_{\tau} \in (\Z^2_+)^{\Hom(F,E)}$  be a pure dominant weight such that $n_\tau \equiv n_{\tau c} \text{ (mod } 2)$ and such that for any $\tau \in \Hom(F,E)$, $n_\tau$ and $m_\tau$ have the same sign. 
Then there exists a pair $(\mu,\psi)$ where
\begin{itemize}
\item $\mu = (k_\tau , k_\tau)_{\tau \in \Hom(F,E)}$ is a diagonal weight such that $\lambda +\mu$ 
or $(\lambda+\mu)^\vee$ satisfies the condition (PR),

\item and $\psi : G_F \to \O^\times$ is a continuous character 
\end{itemize}

which satisfies conditions (1) and (2) of Proposition \ref{twist} for $m^\prime = (k_\tau)_{\tau \in \Hom(F,E)}$. 
\end{corollary}

\begin{proof}
Let $k_\tau = \frac{n_{\tau c} - n_\tau}{2}$ for any $\tau \in \Hom(F,E)$
 and $\mu^\prime = (k_\tau , k_\tau)_{\tau \in \Hom(F,E)}$. 
 Hence $\lambda+\mu^\prime = (a_\tau,b_\tau)_{\tau \in \Hom(F,E)}$ is a pure dominant weight 
with the same purity weight and such that $a_\tau = a_{\tau c} , b_\tau = b_{\tau c}$; moreover, all these integers have the same sign.

If $a_\tau \leq 0$, put $\mu = (k_\tau -1 , k_\tau -1)_{\tau \in \Hom(F,E)}$, then the weight $\lambda^{(PR)}= \lambda+\mu$ satisfies the condition (PR).

If $a_\tau \geq 0$, put $\mu = (k_\tau +1 , k_\tau +1)_{\tau \in \Hom(F,E)}$, then the weight $ \lambda^{(PR)}=(\lambda+\mu)^\vee$ satisfies the condition (PR).

The existence of $\psi$ follows from Lemma 2.2 of \cite{HSBT}.
\end{proof}

\begin{remark}
Actually, the proof also shows that one can choose $\mu$ such that for $\lambda^{(PR)}$ defined as in the proof above, 
we have $\vert w(\lambda^{(PR)}) \vert = \vert w(\lambda) \vert + 2$, where $w(\mu)$ denotes the purity weight of $\mu$.   
\end{remark}

\subsection{Direct summand factors of $\bar{\rho}_{\partial}$}
Under the assumptions of Lemma \ref{fac}, we can see $\bar{\theta}$ as an $H(\F_q)$-representation. First, we note that by \ref{theta} 
we have a decomposition
$$\wedge^2 
 (\bar{\rho}_{\tilde{\m}}^{\vee} (-5)) \cong {\det}^{-1}(-5) \oplus 
 \det(1)  \oplus \epsilon^{\otimes(-2)} \oplus (\Sym^2)^\vee \otimes \det (-2)$$
 as $\im(\bar{\rho}_\m)$-representations. 
 It follows from this that any irreducible subrepresentation of the restriction of the representation $
  \bar{\theta}=(\otimes \Ind_{F}^{F_0})\wedge^2 
 (\bar{\rho}_{\tilde{\m}}^{\vee} (-5)) $ to $G_{F^\prime}$ is a direct summand of $\otimes_{\tau 
  \in I} \phi_\tau$, where the representation $\phi_\tau$ is 
  either a character or is isomorphic to the twist of $\Sym^2$ 
   by a character as $H(\F_q)$-representation.
  
\begin{lemma} \label{1}
If $p > d^2 (d+1)(-2w+4)$, $\text{gcd}(p-1,w+1)=1$ and $\chi$ is a character which is direct 
summand of $\bar{\rho}_\partial$. Then $\chi$ 
is a direct summand of $\bar{\theta}$ as $G_{F^\prime}
$-representation.
\end{lemma}

\begin{proof}
Since the restriction $\chi_{F^\prime}$ of $\chi$ to $G_{F'}$ factors 
by $H(\F_q)$ and since $\SL_2(\F_q)$ is a simple group, 
the character $\chi_{F^\prime}$ factors into a character of
the quotient $\mathcal{D}$ of $H(\F_q)$.
Therefore, there is an integer $k$ such that $\chi_{F^\prime}$ is isomorphic to $\det^k$ as an $H(\F_q)$-representation. So $\chi_{F^\prime}((M_i)_{i\in I})=a^k$, 
where $M_i = \text{diag}(a,1)$ for $a\in \F_p$ and any $i \in 
I$. We assume that $\chi$ is not a direct summand of $\bar{\theta}$.

  Since $\text{char}_{\bar{\theta}(g)}
  (\bar{\rho}_\partial (g)) = 0$ for any $g
  \in G_{F'}$  and $\chi$ is a character which is direct summand of $\bar{\rho}_\partial$, we see that $
  \text{char}_{\bar{\theta}(g)}(\chi (g))=0$ for any 
  $g\in G_{F'}$, hence $\chi(g)$ is an 
  eigenvalue of $\bar{\theta}(g)$.

If we put $M_i^\prime= \text{diag}(ab,b^{-1})$ for any $i\in I$, then $\chi((M_i')_{i\in I})=a^k$ is an eigenvalue of $\bar{\theta}((M_i')_{i\in I})$. 
 
 The condition $\gcd(p-1,w+1)=1$ implies that there is an element $t\in \Z/(p-1)\Z$ such that $\epsilon \cong \det^t$.
 
 Note that $a^k$ is an eigenvalue of a matrix of the form $\otimes_{\tau 
  \in I} \phi_\tau((M_i)_{i\in I})$ where the representation $\phi_\tau$ is 
  either ${\det}^{-1}\otimes\varepsilon^{-5}$ or 
 $\det\otimes\varepsilon$ or $\epsilon^{-2}$ or $(\Sym^2)^\vee \otimes \det\otimes\varepsilon^{-2}$.
Hence, using $\epsilon \cong \det^t$, we see that the eigenvalues of $\phi_\tau((M_i)$ are respectively $a^{-1}\cdot a^{-5t}$, $a\cdot a^t$, $a^{-2t}$ or they belong to
$\{a^{-1}\cdot a^{-2t}\cdot b^{-2},a\cdot a^{-2t}\cdot b^2,a^{-2t}\}$.
 Hence $a^k$ is of the form  $a^j b^l$ where $-2d \geq l \geq 2d$ is a non-zero even number. There are only $(-2w+4)d$ possible values for 
$j$, because $(w+1). \{ -1-5t, 1+t, -2t, -1-2t,1-2t,-2t\}= \{-w+4,w,2,-w+1,w+3,2\}$ (note that $t (w+1) =-1 \text{ (mod }p-1)$).

Since the equation $a^k= a^j x^l \text{ (mod } p)$ has at most $l$ roots and $p > 
d^2 (d+1) (-2w+4)$, there must exist an element $b\in \F_p$ such that $a^k$ does not 
belong to the set of eigenvalues of  
$\bar{\theta}((M_i')_{i\in I})$. Contradiction. 
 Thus, $\chi$ is a direct summand of $\bar{\theta}$ as 
 $G_{F_0}$-representation.

\end{proof}

\begin{remark} \label{bound}
The equation $a^k= a^j x^l \text{ (mod } p)$ has no root or it has exactly $\text{gcd}(p-1,l)$ roots. Therefore, the bound $d^2 (d+1)(-2w+4)$ can be improved as

$$2 d \cdot (\sum_{i=1}^{d} \text{gcd}(p-1,i)) (-2w+4)$$
which unfortunately depends on $p$.
\end{remark}  

By the above Lemma, the character $\chi$ is equal to $\bar{\psi}$, which is a direct summand 
 of $\bar{\theta}$ as $G_{F'}$-representation, so by the Frobenius reciprocity
 theorem, we have a nonzero $G_{F_0}$-equivariant map $\Ind_{F'}
 ^{F_0} \psi \rightarrow \chi$.
 
 By Mackey theorem, there is a one-dimensional subquotient in 
 the $\Ind_{F'}^{F_0} \psi$ if and only if for any $\tau \in \Gal 
 (F'/F_0)$ we have $\psi \cong \psi^{\tau}$. This implies that $
 \chi$ is isomorphic to $(\otimes \Ind_{F}^{F_0})
((\wedge^2\bar{\rho}_{\m}^{\vee})(-5))$ or $(\otimes \Ind_{F}
^{F_0}) ((\wedge^2 
\bar{\rho}_{\m}^{c})(1))$ or $(\otimes \Ind_{F}^{F_0})     
(\epsilon^{\otimes -2})$ as $G_{F'}$ representation.
Thus by lemma \ref{inertia} they have the same Fontaine-Laffaille 
weights.

Therefore the  Fontaine-Laffaille weight of $\chi$ is equal to 
$d(-w+4)$ or $d(w+1)$ or $2d$.

\subsection{Proof of the main Theorem}

 \begin{lemma} \label{2}
 $\chi$ can not have  Fontaine-Laffaille weight $d(-w+4)$ or $2d$.
 \end{lemma}
 
\begin{proof}
The character $\chi$ occurs as a subrepresentation of $H_{\partial}^i (\tilde{X}
_{\tilde{K}}, \V_{\tilde{\lambda}}/ \varpi)$; by Formula (\ref{ls2}),
$\chi$ is also a subrepresentation of $ H_{c}^{i+1} (\tilde{X}
_{\tilde{K}}, \V_{\tilde{\lambda}}/ \varpi)$. 

By theorem \ref{LP}, Fontaine-Laffaille weights of 
$ H_{c}^{i+1} (\tilde{X}_{\tilde{K}}, \V_{\tilde{\lambda}}/ 
\varpi)$ form a subset of $\{ p(\widetilde{w})\vert \widetilde{w} \in W^P , \ell (\widetilde{w}) \leq 
i+1 \}$ where $p(\widetilde{w})=  -\widetilde{w}(\rho + \tilde{\lambda})(H) + \rho(H)
$ ($H = \text{diag} (0,0, ...,0,-1,-1,...,-1)$).

The only representation of $d(-w+4)$ as $p(\widetilde{w})$ is $p(\widetilde{w}_0)$, where 
$\widetilde{w}_0 \in W^P$ is the longest length element. Since $\ell(\widetilde{w}_0) =4d$ and $i<d
$, then $d(-w+4)$ is not Fontaine-Laffaille weight of 
$ H_{c}^{i+1} (\tilde{X}_{\tilde{K}}, \V_{\tilde{\lambda}}/ 
\varpi)$ and also it is not Fontaine-Laffaille weight of 
$\chi$.

For any $\widetilde{w} = \otimes_{\bar{\tau} \in  \tilde{I}} \widetilde{w}_{\bar{\tau}} 
\in W^P$ we have $p(\widetilde{w}) = \sum_{\bar{\tau} \in  \tilde{I}} 
p(\widetilde{w}_{\bar{\tau}})$ and $\ell(\widetilde{w}) = \sum_{\bar{\tau} \in  
\tilde{I}} \ell(\widetilde{w}_{\bar{\tau}})$.  The set $W_{\bar{\tau}}^P$ has 
six elements and every element corresponds to a subset of $
\{1,2,3,4\}$ with two elements.
\\

\begin{tabularx}{0.8\textwidth} { 
  | >{\raggedright\arraybackslash}X 
  | >{\centering\arraybackslash}X 
  | >{\raggedleft\arraybackslash}X | }
 \hline
 Subset of $\{1,2,3,4\}$ & $\ell(\widetilde{w})$ & $p(\widetilde{w})$ \\
 \hline
 $\{1,2\}$  & $0$  & $w$  \\
\hline
$\{1,3\}$ &  $1$ & $-m_\tau + n_\tau +1$\\
\hline 
$\{1,4\}$ & $2$ & $2$ \\
\hline
$\{2,3\}$ & $2$ & $2$\\
\hline
$\{2,4\}$ & $3$ & $-n_\tau +m_\tau +3$\\
\hline
$\{3,4\}$ & $4$ & $-w +4$ \\
\hline
\end{tabularx}
\\

We assume that there is $\widetilde{w} \in W^P$ such that $p(\widetilde{w})=2d$ and $\ell  
(\widetilde{w}) \leq d$. Since $p(\{1,2\}) + p(\{3,4\}) = p(\{1,3\}) + 
p(\{2,4\})$ and $\ell(\{1,2\}) + \ell(\{3,4\}) = \ell(\{1,3\}) + 
\ell(\{3,4\}),$ we can assume that there is no $\bar{\tau} \in 
\Hom(F^+,E)$ such that $w_{\bar{\tau}} $ corresponds to $\{3,4\}$ or 
there is no $\bar{\tau} \in \Hom(F^+,E )$ such that $w_{\bar{\tau}} $ 
corresponds to $\{1,2\}$. The second case is impossible because $
\sum_{\bar{\tau} \in  J_{F^{+}}} \ell(w_{\bar{\tau}}) \leq d$ and 
$-m_\tau + n_\tau +1 \neq	2$.

The inequality $\sum_{\bar{\tau} \in  \Hom(F^+,E)} 
\ell(\widetilde{w}_{\bar{\tau}}) \leq d$ implies that the cardinality of $\{ 
\bar{\tau} \in J_{F^+} \vert \widetilde{w}_{\bar{\tau}} \leftrightarrow 
\{1,2\} \}$ is larger than the cardinality of
$\{ \bar{\tau} \in J_{F^+} \vert \widetilde{w}_{\bar{\tau}} \nleftrightarrow 
\{1,2\} , \{1,3\} \}$.

But $w + 2, w-n_\tau +m_\tau +3, w-w+4$ all is less than or equal 
to 4, then $p(\widetilde{w}) = \sum_{\bar{\tau} \in  \Hom(F^+,E)} 
p(\widetilde{w}_{\bar{\tau}}) \leq 2d$ and we have equality only when there 
is not $\bar{\tau} \in  \Hom(F^+,E)$ such that it corresponds to $
\{1,2\}$, but it is impossible as above. Therefore $2d$  is not 
Fontaine-Laffaille weight of 
$ H_{c}^{i+1} (\tilde{X}_{\tilde{K}}, \V_{\tilde{\lambda}}/ 
\varpi)$ and also it is not Fontaine-Laffaille weight of 
$\chi$.

\end{proof}

Note that for every finite generated torsion $T^S$-module, resp, $\tilde{T}^S$-module $V$, we have $V[\m]=0$, resp, $V[\tilde{\m}]=0$ if and only if $V_\m=0$, resp, $V_{\tilde{\m}}=0$.

  \begin{lemma} \label{3}
	If $i<d$, then $d\cdot w$ is not a Fontaine-Laffaille 
	weight of  $ H_{c}^{i+1} (\tilde{X}_{\tilde{K}}, 
	\V_{\tilde{\lambda}}/ 
\varpi)_{\tilde{\m}}$.
  \end{lemma}

  \begin{proof}
  The integer $d\cdot w$ is equal to $p(\Id)$ where $\Id$ is the trivial element of $
W^P$.  By Corollary \ref{LP1}, the multiplicity of the weight $d\cdot w
$ in $ H_{c}^{i+1} (\tilde{X}_{\tilde{K}}, \V_{\tilde{\lambda}}/ 
\varpi)_{\tilde{\m}}$ is equal to the $\kappa_\p$-dimension of $ H^{i+1} 
(\tilde{\X}_{\tilde{K},\p}^{\text{tor}}, 
 \W_{\tilde{\lambda},\p}^{\text{sub}})_{\tilde{\m}} $, 
where $\tilde{\X}_{\tilde{K},\p}^{tor}$ is a smooth toroidal compactification 
of $\tilde{\X}_{\tilde{K},\p}^{tor} \times \kappa_\p$ and $\W_{\tilde{\lambda},\p}^{\text{sub}(\text{can})}$
 is the sub-canonical (canonical) extension 
of $\W_{\tilde{\lambda},\p} := \mathscr{E}_G(\V_{\lambda_w})$.

By Serre duality as Proposition \ref{Ser}, 
$$ H^{i+1} 
(\tilde{\X}_{\tilde{K},\p}^{\text{tor}}, 
\W_{\tilde{\lambda},\p}^{\text{sub}})_{\tilde{\m}} \cong
  H^{4d-i-1} 
(\tilde{\X}_{\tilde{K},\p}^{\text{tor}}, 
\W_{-2\rho_{nc}- w_{0,G} \tilde{\lambda},\p}^{\text{can}})_{\tilde{\m}^\lor}^\lor$$
 as $
\kappa_\p$ vector space, where $2\rho_{nc}:=((2,2),
(-2,-2))_{\bar{\tau} \in \tilde{I}}$ and where $w_{0,G} \in W_G$ is the
longest length element. Then
 $-2\rho_{nc}- w_{0,G} \tilde{\lambda} = ((m_\tau - 2 , n_\tau 
 -2),(-n_\tau + 2,- m_\tau + 2))_{\bar{\tau} \in \tilde{I}}$. Since the 
 multiplicity of $d (w-4)$ 
as Fontaine-Laffaille weight in $ H^{4d -
 i-1} (\tilde{X}_{\tilde{K}}, \V_{\tilde{\mu}}/ 
\varpi)_{\tilde{\m}^\lor}$ is equal to $\kappa_\p$-dimension of $ H^{4d-i-1} 
(\tilde{X}_{\tilde{K},\p}^{\text{tor}}, 
\W_{-2\rho_{nc}- w_{0,G} \tilde{\lambda},\p}^{\text{can}})_{\tilde{\m}^\lor}$ where  
$\tilde{\mu} := (-n_\tau +2 , -m_\tau+2, m_\tau -2, n_\tau -
2)_{\bar{\tau}\in \tilde{I}}$, because $|\tilde{\mu}|_{comp} = |\tilde{\lambda}|_{comp} +4d < p$. Then by Theorem \ref{cs}, 
we have $ H^{4d -i-1} (\tilde{X}_{\tilde{K}}, \V_{\tilde{\mu}}/ 
\varpi)_{\tilde{\m}^\lor} = 0$ and therefore by corollary \ref{LP1} we have 
$$ H^{4d-i-1} 
(\tilde{X}_{\tilde{K},\p}^{\text{tor}}, 
\W_{-2\rho_{nc}- w_{0,G} \tilde{\lambda},\p}^{\text{can}})_{\tilde{\m}^\lor} = 0$$.

 \end{proof}
 
 Let us define another condition
\begin{itemize}
	\item[($p$-small)] $p \geq 2d^2 (d+1) (-2w+4)$ and $w$ is even. 
\end{itemize}

 \begin{theorem}\label{main}
 Let $\pi$ be a regular, cuspidal and cohomological automorphic 
 representation of $\GL_2 (\A_F)$ with level $K$ and weight $
 \lambda$ such that $K_p$ is hyperspecial. Assume $\pi$ is unramified 
 outside of $S$ and $\lambda$ satisfies conditions (p-small) and assumption of Corollary \ref{weight}. Let $\m \subset \T^S$  
be the associated maximal 
 ideal of Hecke algebra. Assume that (LI $
 \bar{\rho}_\m$) holds. then $H^i (X_K, \V_\lambda/\varpi)_\m$ is 
 zero for all $i \notin [d,3d-1]$.
\end{theorem}  
  
  \begin{proof}
First of all by Proposition \ref{poincare}, it suffices to prove vanishing of the cohomology for only $i<d$ and the Corollary \ref{weight} implies that we can assume that $\lambda$ satisfies the condition (PR).

  The ($p$-small) condition guaranties that there is a twist of $\pi$ with pure weight $w'$ such that satisfies $p \geq d^2 (d+1) 
  (-2w'+4)$ and $\gcd (w'+1 , p-1) = 1$. Therefore by theorem \ref{twist}, 
  we can assume that $p \geq d^2 (d+1) (-2w+4)$ and $\gcd (w+1 , 
  p-1) = 1$. So lemmas \ref{1}, \ref{2} and \ref{3} implies that 
  $H_{\partial}^i (\tilde{X}_{\tilde{K}}, \V_{\tilde{\lambda}}/ 
  \varpi)_{\tilde{\m}} = 0$ for all $i<d$. Hence theorem 
  \ref{acc1} implies that $H^i (X_K, \V_{\lambda}/ \varpi)_{m} = 0$ for 
  all $i < d$. 
  \end{proof}
  
  \begin{remark}
  As Remark \ref{bound} we can improve the bound to 
  $$ 4 d \cdot (\sum_{i=1}^d \text{gcd}(p-1,i)) (2|w|+4)$$ 
  \end{remark}

  \begin{corollary}\label{cormin}
  Under the assumptions of Theorem \ref{main}, we have 
  $H^d(X_K,\V_\lambda)_\m$ is torsion-free.
 \end{corollary}

\begin{proof}
It follows from the following long exact sequence:
$$\cdots \rightarrow H^{i-1}(X_K,\V_\lambda/\varpi)_\m \rightarrow H^i(X_K,\V_\lambda)_\m \xrightarrow{\times \varpi} H^i(X_K,\V_\lambda)_\m \rightarrow H^i(X_K,\V_\lambda/\varpi)_\m.
$$
\end{proof}

Recall that Calegari and Geraghty introduced Conjectures A (see \cite[Section 5.3]{CG18}) and B (see \cite[Section 9.1]{CG18}) about the existence of Galois representations over localized Hecke algebras $\T^S$ and $\T^S_Q$ with characteristic polynomials given by Hecke polynomials at primes outside $p$ and the level, and satisfying local global compatibilities,  which we denote by (LGC), at primes in the level or above $p$. 
Furthermore they assume in Conjecture B that $H^i(X_K,\V_\lambda/\varpi)_\m=0$,
for $i\notin [q_0,q_0+\ell_0]$. This is our Main Theorem.

If $p$ splits completely in $F$, the existence of the Galois representations over $\T^S$ and $\T^S_Q$ with the expected characteristic polynomial at primes outside the level and $p$
has been established in \cite[Theorem 6.1.4]{CGH}. Moreover (LGC) holds at primes in $Q$ by \cite[Lemma 9.6]{CG18} since $n=2$. It holds  at $p$ in the Fontaine-Laffaille case up to an nilpotent ideal by \cite[Theorem 4.5.1]{ACC}.

Let us assume that the level group $K$ is of Iwahori type $K=K_0(\n)$ and that
 $\bar{\rho}_\m$ is minimal at primes dividing $\n$ (see \cite[Definition 3.1 (4)]{CG18}). Recall that (LGC)
holds at primes dividing $\n$ modulo a nilpotent ideal by \cite[Theorem 3.1.1]{ACC}. 

\begin{corollary}\label{CG} 
  Under the assumptions of Theorem \ref{main}, if we assume  $\bar{\rho}_\m$ minimal and that (LGC) holds at primes dividing $\n$, then
  $H^{3d-1}(X_K,\V_\lambda)_\m$ is free of rank one over $\T^S$ and
$H^{d}(X_K,\V_\lambda)_\m$ is isomorphic to $\Hom(\T^S,\O)$ as $\T^S$-module.
 \end{corollary}

\begin{proof}
By the assumptions and our Main Theorem, Conjecture B holds. Hence the result follows from \cite[Theorems 6.3 and 6.4]{CG18} . The rank is one because it is so if we localize  the cohomology at each automorphic representation occuring in $\T^S$ by calculations of Borel-Wallach, Chapter 3. The second part follows by Poincar\'e duality applied to
 $H^{3d-1}(X_K,\V_\lambda)_{\m^\vee}$.
\end{proof}

Shayan Gholami, gholami@math.univ-paris13.fr, LAGA UMR 7539, Institut Galil\'ee, USPN Universit\'e Paris XIII, 93430 Villetaneuse, France

\end{document}